\documentclass[11pt]{article}
\usepackage[a4paper, total={6in, 9in}]{geometry}
\usepackage{amsmath} 
\usepackage{mathtools} %
\usepackage{amssymb} 
\usepackage{graphicx}
\usepackage{enumitem}
\usepackage[font=small,labelfont=bf]{caption}
\usepackage{subcaption}
\usepackage{amsthm}
\usepackage{siunitx}
\usepackage{comment}
\usepackage{authblk}
\usepackage{mleftright}
\newtheoremstyle{break}
  {\topsep}{\topsep}%
  {\itshape}{}%
  {\bfseries}{}%
  {\newline}{}%
\theoremstyle{break}
\newtheorem{theorem}{Theorem}[section]

\newtheorem{proposition}[theorem]{Proposition}
\newtheorem{lemma}[theorem]{Lemma}
\newtheorem{remark}[theorem]{Remark}

\newtheorem{definition}[theorem]{Definition}
\usepackage{hyperref}
\hypersetup{colorlinks = true, linkcolor = blue, filecolor = magenta, urlcolor = cyan, citecolor = blue, pdfpagemode = FullScreen}

\usepackage[
    backend=biber,        
    style=numeric,        
    sorting=nyt,          
    maxnames=9,           
    giveninits=true,      
    doi=true,             
    url=false,            
    isbn=false,           
    sortcites=true,
]{biblatex}

\DeclareNameAlias{sortname}{family-given} 
\DeclareNameAlias{default}{family-given}

\DeclareFieldFormat{doi}{%
  \ifhyperref
    {\href{https://doi.org/#1}{DOI: \nolinkurl{#1}}}
    {DOI: \nolinkurl{#1}}}

\usepackage{orcidlink}

\DeclareMathOperator{\supp}{supp}
\DeclareMathOperator{\vspan}{span}
\DeclareMathOperator{\dist}{dist}

\title{Directional polynomial frames on spheres}

\author[1]{Marzieh Hasannasab \orcidlink{0000-0002-3975-5545}\,\thanks{\href{mailto:mhas@dtu.dk}{mhas@dtu.dk}}}
\author[2]{Larissa Kaldewey \orcidlink{0009-0006-6341-5635}\,\thanks{\href{mailto:larissa.kaldewey@student.uni-luebeck.de}{larissa.kaldewey@student.uni-luebeck.de}}}
\author[2]{Frederic Schoppert \orcidlink{0000-0002-8682-3723}\,\thanks{\href{mailto:f.schoppert@uni-luebeck.de}{f.schoppert@uni-luebeck.de}}}
\affil[1]{Department of Applied Mathematics and Computer Science, Technical University of Denmark, Building 303, 2800 Lyngby, Denmark}
\affil[2]{Institute of Mathematics, University of L{\"u}beck,
Ratzeburger Allee 160, 23562, L{\"u}beck, Germany}

\date{\vspace{-5ex}}

\begin{document}

\maketitle
\begin{abstract}
We introduce a general framework for the construction of polynomial frames in $L^2(\mathbb{S}^{d-1})$, $d \geq 3$, where the frame functions are obtained as rotated versions of an initial sequence of polynomials $\Psi^j$, $j\in \mathbb{N}_0$. The rotations involved are discretized using suitable quadrature rules. This framework includes classical constructions such as spherical needlets and directional wavelet systems, and at the same time permits the systematic design of new frames with adjustable spatial localization, directional sensitivity, and computational complexity. We show that a number of frame properties can be characterized in terms of simple, easily verifiable conditions on the Fourier coefficients of the functions $\Psi^j$. Extending an earlier result for zonal systems, we establish sufficient conditions under which the frame functions are optimally localized in space with respect to a spherical uncertainty principle, thus making the corresponding systems a viable tool for position-frequency analyses. To conclude this article, we explicitly discuss examples of well-localized and highly directional polynomial frames.
\end{abstract}

\noindent\textbf{Keywords:} Polynomial frames; Frame constructions; Localized frames; Sphere; Directional wave\-lets; Curve\-lets.

\section{Introduction}
This paper aims to develop a general framework for constructing polynomial frames for $L^2(\mathbb{S}^{d-1})$, $d\geq 3$, that covers both zonal and directional systems and enables the design of new frames with tunable spatial localization, directional sensitivity, and computational complexity. Our approach unifies many existing constructions, most prominently spherical needlets and directional wavelet systems, within a single setting. It provides a flexible procedure for generating new frames that span the full range from isotropic to highly anisotropic designs.

Polynomial frames for $L^2(\mathbb{S}^{d-1})$, $d\geq 3$, have been studied predominantly in the zonal (isotropic) setting, where each frame function is invariant under rotations around some axis (see, e.g., \cite{bib11, bib3, bib15, bib16, bib29, bib30, bib14, bib46, bib47, bib48, bib49, bib25}). A recurring theme in this context is the pursuit of good localization properties, as these systems are typically constructed to enable position-based analyses of signals. 
A general initial investigation, including a discussion of systems that are optimally concentrated in space with respect to a spherical uncertainty principle \cite{bib42, bib45}, is given in \cite{bib11}.
Subsequent work \cite{bib3, bib15, bib16, bib78} established that certain zonal polynomials with particularly smooth spectral behavior satisfy sharp pointwise localization bounds, yielding rapid decay away from the center of mass. These functions are commonly referred to as spherical needlets. Owing to their excellent spatial localization, needlet-based frames exhibit rapid convergence of frame expansions in the spaces $L^p(\mathbb{S}^{d-1})$, $1 \leq p < \infty$, and $C(\mathbb{S}^{d-1})$. Moreover, such frames have been successfully applied to the characterization of function spaces \cite{bib30, bib48, bib25} and to edge detection in spherical signals \cite{bib14, bib26, bib31}.

Beyond these applications, spherical needlets provide a powerful tool for position-fre\-quency analysis and have found widespread use in the statistical analysis of spherical random fields (see, e.g., \cite{bib55, bib46, bib56, bib57, bib58, bib59, bib60, bib61, bib62}), particularly in cosmology and astrophysics (see, e.g., \cite{bib50, bib51, bib52, bib53, bib54}). Additional areas of relevance include sparse image reconstruction (e.g., \cite{bib66}) and image segmentation (e.g., \cite{bib68}).

Directional (non-zonal) systems, by contrast, are essential when anisotropic features must be analyzed. Constructions such as directional wavelets and second-generation curve\-lets \cite{bib9, bib10, bib21, bib1} can identify edges and higher-order singularities with precise position and orientation information. They have been used in the analysis of random fields (e.g., \cite{bib9, bib70, bib63, bib65, bib67, bib69}), with particular emphasis on problems in cosmology, but also general issues in image processing such as sparse reconstruction (e.g., \cite{bib66}) and segmentation (e.g., \cite{bib68}).
On higher-dimensional spheres $\mathbb{S}^{d-1}$, $d\geq 4$, localized directional polynomial frames have only recently been considered in \cite{bib37} in the form of directional wavelets which generalize the two-dimensional constructions in \cite{bib9, bib10, bib21}.

Despite the abundance of results on zonal and directional polynomial frames, existing constructions are typically tailored to specific objectives: zonal frames provide excellent spatial localization but no directional selectivity, whereas directional systems capture orientation information at the expense of increased complexity. Directional wavelet systems, which generalize zonal constructions, are typically built as Parseval frames with a fixed and inherently limited notion of directionality. Second-generation curvelets, on the other hand, can achieve unlimited directional resolution, but they have so far only been developed in two dimensions. In particular, no polynomial frame is currently known on higher-dimensional spheres that provides unlimited directional sensitivity. Our work addresses these limitations by introducing a unified framework  valid across all dimensions, spanning the full spectrum from isotropic to highly anisotropic frames, and provides explicit criteria for both frame properties and localization behavior. 

For our construction, we fix a sequence of nonnegative integers $(N_j)_{j=0}^\infty \subset \mathbb{N}_0$ and consider associated positive quadrature rules
\begin{equation}\label{quadrature introduction}
		\int_{SO(d)} f_1 f_2 \, \mathrm{d}\mu_d = \sum_{r=1}^{r_j} \mu_{j, r}  f_1(g_{j, r}) f_2(g_{j, r}), \quad \text{for } f_1, f_2 \in \Pi_{N_j}(SO(d)), \qquad j \in \mathbb{N}_0.
\end{equation}
Here, $\mu_d$ denotes the normalized Haar measure on $SO(d)$ and the $g_{j, r}\in SO(d)$ are suitable rotation matrices with corresponding weights $\mu_{j, r}>0$. Given a sequence of polynomials $(\Psi^j)_{j=0}^\infty$ with $\Psi^j \in \Pi_{N_j}(\mathbb{S}^{d-1})$, we define the associated system
\begin{equation}\label{frame introduction}
    \mathcal{X}[(\Psi^j)_{j=0}^\infty] = \{\sqrt{\mu_{j, r}}\,\Psi^j(g_{j, r}^{-1} \,\cdot) \mid r=1, \dots , r_j, \; j \in \mathbb{N}_0\}.
\end{equation}
This construction provides a flexible framework in which the key  properties of the resulting system are determined entirely by the choice of the polynomial sequence $(\Psi^j)_{j=0}^\infty$. By selecting zonal polynomials, one recovers systems with classical isotropic behavior, whereas more general choices yield directional or partially directional systems capable of capturing anisotropic features. 

Considering a classical orthonormal basis for $L^2(\mathbb{S}^{d-1})$ consisting of spherical harmonics, we show that the existence of frame bounds for the system \eqref{frame introduction} can be characterized in terms of explicit and easily verifiable conditions on the Fourier coefficients of the generating polynomial sequence $(\Psi^j)_{j=0}^\infty$. 
Moreover, we demonstrate that the canonical dual frame can be constructed directly from the Fourier coefficients of $(\Psi^j)_{j=0}^\infty$ and retains the same structural form as the original frame. Finally, we provide a characterization, also in terms of the Fourier coefficients, under which two frames are dual.

We also discuss relevant properties such as rotational symmetries, directional sensitivity, and steerability and demonstrate how they can be obtained. Furthermore, extending a result of \cite{bib11} for the zonal systems, we prove that, under relatively mild 
assumptions on the Fourier coefficients of $\Psi^j$, $j\in\mathbb{N}_0$, the frame functions localize at an optimal rate in terms of a well-established spherical uncertainty principle \cite{bib38}. Specifically, if $\mathrm{Var}_{\mathrm{S}}(\Psi^j)$ denotes the spatial variance of $\Psi^j$, the optimal localization rate takes the form
\begin{equation*}
    \mathrm{Var}_{\mathrm{S}}(\Psi^j) \sim N_j^{-2}.
\end{equation*}
We also note that, for non-zonal systems, localization has so far only been studied in the pointwise sense (see \cite{bib9, bib37, bib26, bib31}), which is highly restrictive, as pointwise estimates typically demand strong assumptions.

To conclude our studies, we propose various choices for the sequence $\Psi^j$, $j\in\mathbb{N}_0$ that yield well-localized and highly directional polynomial frames.  These constructions are of considerable interest for practical applications, not least because their high degree of symmetry reduces the number of rotations needed compared to the general case.

The remainder of this article is organized as follows. In \autoref{sec2}, we review essential facts from harmonic analysis on spheres and rotation groups, including a key product property for polynomials on $SO(d)$ that will play a central role in the discretization of our frames. In \autoref{sec3}, we carry out our main investigations regarding the construction and characterization of polynomial frames on spheres. Relevant properties, such as steerability, directionality, and optimal localization in terms of a well-established spherical uncertainty principle, are discussed in \autoref{sec4}, where we also demonstrate how integration rules of the form \eqref{quadrature introduction} can be obtained from spherical quadrature formulas. Finally, in \autoref{sec5}, we showcase a variety of concrete examples of well-localized and highly directional frames.

\section{Preliminaries}\label{sec2}
In this section, we summarize some basic concepts of harmonic analysis on spheres and rotation groups. Most of the formulas given here can be found in classical literature such as \cite{bib39, bib32} (see also \cite{bib3, bib75} for a more recent monograph) and are stated without further comment.

For $d\geq 3$, we consider the euclidean space $\mathbb{R}^d$ equipped with the inner product $\langle x, y \rangle = x^\top y$ and the induced norm $\| \cdot \|$. Its canonical basis vectors will be denoted by $e^j=(\delta_{i, j})_{i=1}^d$, $j=1, \dots, d$, where
\begin{equation*}
\delta_{i, j}= \begin{cases}
1, \quad & i=j,\\
0,  & i \neq j.
\end{cases}
\end{equation*}
The standard metric on the sphere $\mathbb{S}^{d-1}= \{ x \in \mathbb{R}^d \mid \| x\| =1\}$ is given in terms of the geodesic distance
\begin{equation*}
\dist(\eta, \nu) = \arccos(\langle \eta, \nu \rangle), \quad \eta, \nu \in \mathbb{S}^{d-1}.
\end{equation*}
By $\omega_{d-1}$ we will denote the rotation-invariant measure on $\mathbb{S}^{d-1}$, which is normalized such that
\begin{equation*}
\int_{\mathbb{S}^{d-1}} \mathrm{d}\omega_{d-1} = 1. 
 \end{equation*}
With respect to this measure, the spaces $L^p(\mathbb{S}^{d-1})$, $1\leq p < \infty$, are defined by
\begin{equation*}
    L^p(\mathbb{S}^{d-1})= \{ f \colon \mathbb{S}^{d-1} \rightarrow \mathbb{C} \mid f \text{ is measurable and } \int_{\mathbb{S}^{d-1}} |f|^p\,\mathrm{d}\omega_{d-1} < \infty \},
\end{equation*}
endowed with the norm
\begin{equation*}
    \|f \|_{L^p(\mathbb{S}^{d-1})} = \left( \int_{\mathbb{S}^{d-1}} \lvert f \rvert^p \, \mathrm{d}\omega_{d-1}\right)^{1/p}.
\end{equation*}
Also, let $C(\mathbb{S}^{d-1})$ denote the space of continuous functions $f \colon \mathbb{S}^{d-1}\rightarrow \mathbb{C}$, equipped with the supremum norm
\begin{equation*}
    \| f \|_{C(\mathbb{S}^{d-1})} = \max_{\eta \in \mathbb{S}^{d-1}} \lvert f(\eta) \rvert.
\end{equation*}
In this article, we will work mainly in the Hilbert space $L^2(\mathbb{S}^{d-1})$ with the inner product
\begin{equation*}
\langle f_1, f_2 \rangle_{\mathbb{S}^{d-1}} = \int_{\mathbb{S}^{d-1}} f_1  \overline{f_2} \, \mathrm{d}\omega_{d-1}.
\end{equation*}
In $L^2(\mathbb{S}^{d-1})$, Fourier analysis is typically formulated w.r.t.\ a classical orthonormal basis consisting of spherical harmonics. Such a basis can be introduced as follows. Let $\Pi_n(\mathbb{R}^d)$ be the space of all $d$-variate algebraic polynomials of degree $n$. Thus, using the notation
\begin{equation*}
    \lvert \alpha \rvert = \alpha_1+\cdots+\alpha_d, \qquad x^\alpha = x_1^{\alpha_1} \cdots x_d^{\alpha_d},
\end{equation*}
for $\alpha = (\alpha_1, \dots, \alpha_d)\in \mathbb{N}_0^d$, the set $\Pi_n(\mathbb{R}^d)$ contains exactly the functions $P \colon \mathbb{R}^d\rightarrow \mathbb{C}$ of the form
\begin{equation*}
P(x) = \sum_{\lvert \alpha \rvert \leq n} \, c_\alpha x^\alpha, \qquad c_\alpha \in \mathbb{C}.
\end{equation*}
An element $P\in \Pi_n(\mathbb{R}^d) $ is called a homogeneous polynomial of degree $n$ if
\begin{equation*}
P(x) = \sum_{\lvert \alpha \rvert = n} \, c_\alpha x^\alpha,
\end{equation*}
for suitable coefficients $c_\alpha$, and it is called harmonic if $\Delta P = 0$, where $\Delta = \partial_1^2 + \cdots + \partial_d^2 $ denotes the euclidean Laplace operator on $\mathbb{R}^d$. The space of all harmonic homogeneous polynomials of degree $n$ is denoted by  $\mathcal{H}_n^d(\mathbb{R}^d)$. By restricting these functions to $\mathbb{S}^{d-1}$, we obtain the set
\begin{equation*}
\mathcal{H}_n^d = \{ P\vert_{\mathbb{S}^{d-1}} \mid P \in  \mathcal{H}_n^d(\mathbb{R}^d)  \},
\end{equation*}
which is a linear subspace of $L^2(\mathbb{S}^{d-1})$. A function $Y \in \mathcal{H}_n^d$ is called a spherical harmonic of degree $n$. The spaces $\mathcal{H}_n^d$ are finite-dimensional with 
\begin{equation*}
\dim \mathcal{H}_{n}^d = \frac{(2n+d-2) (n+d-3)!}{(d-2)! n!},
\end{equation*}
and spherical harmonics of different degrees are orthogonal, i.e., $\mathcal{H}_m^d \perp \mathcal{H}_n^d$ for $m \neq n$. Indeed, 
\begin{equation*}
L^2(\mathbb{S}^{d-1}) = \bigoplus_{n=0}^\infty \mathcal{H}_n^d.
\end{equation*}
The set of polynomials of degree $N$ on $\mathbb{S}^{d-1}$ will be denoted by
\begin{equation*}
\Pi_N(\mathbb{S}^{d-1})= \{P\vert_{\mathbb{S}^{d-1}} \mid P \in \Pi_N(\mathbb{R}^d)  \}= \bigoplus_{n=0}^N \mathcal{H}_n^d
\end{equation*}
and we have
\begin{equation*}
\dim \Pi_N(\mathbb{S}^{d-1}) = \frac{(2N+d-1)(N+d-2)!}{(d-1)!N!}.
\end{equation*}
In the following, for each $n \in  \mathbb{N}_0$, let $\mathcal{I}_n^d $ be a suitable index set of cardinality $\dim \mathcal{H}_n^d$. Moreover, let
\begin{equation*}
Y_k^{d,n} \colon \mathbb{S}^{d-1} \rightarrow \mathbb{C}, \qquad k \in \mathcal{I}_n^d,
\end{equation*}
be functions such that $\{ Y_k^{d, n} , \; k \in \mathcal{I}_n^d\}$ is an orthonormal basis of $\mathcal{H}_n^d$. Then, it follows from the above discussion that the system $\{ Y_k^{d, n} \mid k \in \mathcal{I}_n^d, \; n \in \mathbb{N}_0 \}$ of spherical harmonics is an orthonormal basis for $L^2(\mathbb{S}^{d-1})$. If $f \in L^2(\mathbb{S}^{d-1})$, we will use the notation $f(n, k) = \langle f, Y_k^{d, n}\rangle_{\mathbb{S}^{d-1}}$ for the corresponding Fourier coefficients. Here, it will be convenient to set $f(n, k)=0$ if $n =-1$ or $k \notin \mathcal{I}_n^d$.

Although many of the results in this article will be formulated in a coordinate-indepen\-dent way, it will sometimes be useful to work with an explicitly defined orthonormal basis of spherical harmonics. Such a basis is presented in \hyperref[appendix A]{Appendix A}.

We note that the spaces $\mathcal{H}_n^d$ can also be defined as the eigenspaces of the usual Laplace-Beltrami operator $\Delta_{\mathbb{S}^{d-1}}$. Indeed, in terms of the basis given above, we have
\begin{equation*}
    \Delta_{\mathbb{S}^{d-1}} Y_k^{d, n} = -n(n+d-2)Y_k^{d, n}.
\end{equation*}

At the heart of many computations involving spherical harmonics lies the well-known addition theorem  
\begin{equation*}\label{addition thm}
\sum_{k \in \mathcal{I}_n^d}  \overline{Y_k^{d, n}(\nu)} \, Y_k^{d, n}(\eta) = \frac{2n+d-2}{d-2}\,  C_n^{\frac{d-2}{2}}(\langle \nu, \eta \rangle), \quad \nu, \eta \in \mathbb{S}^{d-1},
\end{equation*}
which provides a link between harmonic analysis on the sphere and one-dimensional Gegenbauer polynomials
\begin{equation}\label{gegenbauer polynomials}
C_n^\lambda(t) = \frac{(-1)^n \Gamma(\lambda+1/2) \Gamma(n+2\lambda)}{2^n\Gamma(n+\lambda+1/2) n!} (1-t^2)^{1/2-\lambda} \frac{\mathrm{d}^{n}}{\mathrm{d}t^{n}}(1-t^2)^{n+\lambda-1/2}
.
\end{equation}

The sphere $\mathbb{S}^{d-1}$ is intimately connected to the special orthogonal group $$SO(d) = \{ g\in \mathbb{R}^{d\times d}: g^\top = g^{-1}, \; \det g =1  \},$$ since $\mathbb{S}^{d-1}$ can be identified with the homogeneous space $SO(d)/SO(d-1)$. Here, we identify $SO(d-1)$ with the subgroup of all elements $g\in SO(d)$ satisfying $g e^d = e^d$. More generally, we will not distinguish between $SO(m)$, $m\in \{2, \dots, d-1\}$, and the maximal subgroup of $SO(d)$ that keeps the vectors $e^{m+1}, \dots, e^d$ invariant. 

By $\mu_d$ we denote the normalized Haar measure on $SO(d)$.
Each element $g\in SO(d)$ can be written as $g = g_\eta h$ where $h \in SO(d-1)$ and $g_\eta \in SO(d)$ with $g_\eta e^{d} = \eta \in \mathbb{S}^{d-1}$. Of course, this representation is not unique. Nevertheless, independent of the choice of $g_\eta$ and $h$ we have
\begin{equation}\label{eq2}
\int_{SO(d)} f(g) \, \mathrm{d}\mu_d(g) = \int_{\mathbb{S}^{d-1}} \int_{SO(d-1)} f(g_\eta h)  \, \mathrm{d}\mu_{d-1}(h) \, \mathrm{d}\omega_{d-1}(\eta),
\end{equation}
for any $f \in L^1(SO(d))$. In particular, the integral over $SO(d-1)$ on the right-hand side does not depend on the choice of $g_\eta$ due to the invariance of the Haar measure $\mu_{d-1}$ on $SO(d-1)$. In this article, $g_\eta$ will always denote an arbitrary element in $SO(d)$ satisfying $g_\eta e^d = \eta$. We now consider the unitary group representation
\begin{equation*}
g \mapsto T^d(g), \quad T^d(g)f(x) = f(g^{-1}x),
\end{equation*}
of $SO(d)$ on $L^2(\mathbb{S}^{d-1})$. Its subrepresentations on the spaces $\mathcal{H}_n^d$ are irreducible and thus the matrix functions $t_{k, m}^{d,n}\colon SO(d) \rightarrow \mathbb{C}$ given by
\begin{equation}\label{eq28}
t_{k, m}^{d,n}(g) = \langle T^d(g) Y_m^{d,n}, Y_k^{d,n} \rangle_{\mathbb{S}^{d-1}},
\end{equation}
form an orthogonal system of $L^2(SO(d))$ with respect to the inner product
\begin{equation*}
\langle f_1, f_2 \rangle_{SO(d)} = \int_{ SO(d)} f_1  \overline{f_2} \, \mathrm{d}\mu_d.
\end{equation*}
Additionally, it holds that
\begin{equation*}
\int_{SO(d)} \lvert t_{k, m}^{d,n} \rvert^2 \, \mathrm{d}\mu_d = \frac{1}{\dim \mathcal{H}_n^{d}}.
\end{equation*}
The definition given in \eqref{eq28} can also be formulated as
\begin{equation}\label{eq30}
Y_m^{d, n}(g^{-1}\eta)=\sum_{k \in \mathcal{I}_n^d} t_{k, m}^{d, n}(g) \, Y_k^{d, n}(\eta), \quad \eta \in \mathbb{S}^{d-1}.
\end{equation}
Moreover, the relation
\begin{equation*}\label{matrix fct prop1}
t_{k, \ell}^{d, n}(g \tilde{g}) = \sum_{m \in \mathcal{I}_n^d} t_{k, m}^{d, n}(g) \, t_{m, \ell}^{d, n}(\tilde{g}), \quad g, \tilde{g} \in SO(d),
\end{equation*}
follows from the fact that $T^d$ is a group homomorphism, i.e., $T^d(g \tilde{g}) = T^d(g)T^d(\tilde{g})$, with each $T^d(g)$ being a unitary operator on $L^2(\mathbb{S}^{d-1})$. In the case where $\eta = e^d$, a simple calculation using \eqref{eq30} yields
\begin{equation*}\label{eq7}
\overline{Y_m^{d, n}(\nu)} = \sum_{k \in \mathcal{I}_n^d} \overline{Y_k^{d, n}(e^d)} \, t_{m, k}^{d, n}(g_\nu).
\end{equation*}
This shows that spherical harmonics are linear combinations of certain matrix functions as, clearly, $\overline{t_{m, k}^{d, n}}$ is a linear combination of the $t_{r, s}^{d, n}$, $r, s \in \mathcal{I}_n^d$. 

Since the elements of the special orthogonal group are $d\times d$ rotation matrices, we can naturally identify $SO(d)$ with a subset of $\mathbb{R}^{d^2}$. In this context, let
\begin{equation*}
\Pi_N(SO(d)) = \Pi_N(\mathbb{R}^{d^2})\vert_{SO(d)}
\end{equation*}
denote the set of all restrictions of polynomials of degree $N$ to $SO(d)$. Using just the definition \eqref{eq28}, it is straightforward to verify that each matrix function $t_{k, m}^{d, n}$ with $n \leq N$ is contained in $\Pi_N(SO(d))$. In particular, defining
\begin{equation*}
\mathcal{M}_N^1(SO(d)) = \vspan\{ t_{k, m}^{d, n}, \; k, m \in \mathcal{I}_n^d, \; 0 \leq n \leq N \},
\end{equation*}
it holds that $\mathcal{M}_N^1(SO(d)) \subset \Pi_{N}(SO(d))$ and, thus, we have the product property
\begin{equation}\label{eq33}
f_1 f_2 \in \Pi_{2N}(SO(d)), \qquad \text{if }f_1, f_2 \in \mathcal{M}_N^1(SO(d)).
\end{equation}
The upper index $1$ in the notation $\mathcal{M}_N^1(SO(d))$ indicates that the space is spanned by matrix functions of class $1$, see \cite{bib32} for more details. For notational convenience, we will define
\begin{equation*}
    \mathcal{M}_N^1(SO(2)) = \Pi_N(SO(2))=\vspan\{\exp(\mathrm{i}n \cdot ), \;  n=-N,\dots, N \},
\end{equation*}
such that \eqref{eq33} also holds for $d=2$.

When working with the specific orthonormal basis of spherical harmonics given in \hyperref[appendix A]{Appendix A}, there is a particularly straightforward connection between matrix functions of different dimensions, as the following lemma shows. For a detailed proof, we refer to \cite[Lemma~2.1]{bib37}.
\begin{lemma}\label{lemma4}
With respect to the orthonormal basis of spherical harmonics given in \hyperref[appendix A]{Appendix A}, the matrix functions satisfy
\begin{equation}\label{matrix fct prop2}
t_{k, \ell}^{d, n}(h) = \delta_{k_1, \ell_1} \, t_{(k_2, \dots, k_{d-2}), (\ell_2, \dots, \ell_{d-2})}^{d-1, k_1}(h), \qquad h \in SO(d-1), \quad d \geq 4,
\end{equation}
as well as
\begin{equation}\label{matrix fct prop2.1}
t_{k, \ell}^{3, n}(h(\gamma)) = \delta_{k, \ell}  \exp(\mathrm{i}k\gamma), \qquad \gamma \in [0, 2\pi),
\end{equation}
where $h(\gamma)\in \mathbb{R}^{3\times 3}$ is a positive rotation by $\gamma$ in the $(x_1, x_2)$-plane.
\end{lemma}
\noindent Equations \eqref{matrix fct prop2} and \eqref{matrix fct prop2.1} imply, in particular, that
\begin{equation*}
    f\vert_{SO(m)} \in \mathcal{M}_N^1(SO(m)) \qquad \text{for each }f \in \mathcal{M}_N^1(SO(d)), \; m \in \{2, \dots, d-1\}.
\end{equation*}

The construction of our polynomial frames in \autoref{sec3} relies heavily on the following product property for class $1$ matrix functions on $SO(d)$.
\begin{lemma}\label{lemma product property}
Let $f_1 \in \mathcal{M}_{N_1}^1(SO(d))$, $f_2 \in \mathcal{M}_{N_2}^1(SO(d))$. Then the function
\begin{equation*}
\eta \mapsto \int_{SO(d-1)} f_1(g_\eta h)f_2(g_\eta h) \, \mathrm{d}\mu_{d-1}(h)
\end{equation*}
is in $\Pi_{N_1+N_2}(\mathbb{S}^{d-1})$.
\end{lemma} 

For $d =3$ the above result is well known. Indeed, in this case the matrix functions $t_{\ell, k}^{3, n}$, often called Wigner $D$-functions, satisfy the product property
\begin{equation*}
    t_{\ell, k}^{3, n} \,  t_{\ell', k'}^{3, n'} \in \mathcal{M}_{n+n'}^1(SO(3)).
\end{equation*}
More details in this direction can be found in \cite[Chapter~4]{bib20}. In the case $d\geq 4$, a proof was recently given in \cite[Lemma~2.2]{bib37}.

\section{Polynomial Frames}\label{sec3}
For the construction of polynomial frames for $L^2(\mathbb{S}^{d-1})$ we fix an increasing sequence $(N_j)_{j=0}^\infty \subset \mathbb{N}_0$ as well as a corresponding sequence of positive quadrature rules
\begin{equation}\label{eq35}
		\int_{SO(d)} f_1 f_2 \, \mathrm{d}\mu_d = \sum_{r=1}^{r_j} \mu_{j, r}  f_1(g_{j, r}) f_2(g_{j, r}) \quad \text{if } f_1, f_2 \in \mathcal{M}_{N_j}^1(SO(d)), \qquad j \in \mathbb{N}_0,
\end{equation}
with suitable points $g_{j, r}\in SO(d)$ and weights $\mu_{j, r}>0$. Although the literature that specifically discusses numerical integration on $SO(d)$, $d>3$, is scarce, there are results for more general domains that are applicable. We refer to \cite{bib28, bib71}, in which the existence and computability of such points $g_{j, r}$ and weights $\mu_{j, r}$ are demonstrated. Moreover, in \autoref{subsec4.5} we will discuss how formulas of the form \eqref{eq35} can be obtained from spherical quadrature rules using \hyperref[lemma product property]{Lemma~\ref*{lemma product property}}. These discretization rules are very natural in this context, and they extend the point sets usually used in the zonal setting.

In the following, let $(\Psi^j)_{j=0}^\infty$ be a sequence of polynomials such that $\Psi^j \in \Pi_{N_j}(\mathbb{S}^{d-1})$, i.e.,
\[
\Psi^j = \sum_{n=0}^{N_j}\sum_{k \in \mathcal{I}_n^d} \Psi^j(n,k) Y_k^{d, n}, \quad j \in \mathbb{N}_0. 
\]
The objects of study in this section are families of the form
\begin{equation}\label{ourFrames}
       \mathcal{X}[(\Psi^j)_{j=0}^\infty] = \{ \sqrt{\mu_{j, r}} \,T^d(g_{j, r})\Psi^j , \; r=1, \dots, r_j, \; j \in \mathbb{N}_0 \}, 
\end{equation}
where we will identify $\mathcal{X}[(\Psi^j)_{j=0}^\infty]$ with the (canonically) ordered sequence
\begin{equation*}
    \left( \Psi^{0, 1}, \dots, \Psi^{0, r_0}, \Psi^{1, 1},\dots, \Psi^{1, r_1}, \Psi^{2, 1}, \dots  \right), \qquad \text{where } \Psi^{j, r} = \sqrt{\mu_{j, r}}\,T^d(g_{j, r})\Psi^j.
\end{equation*}
The sequence $(\Psi^j)_{j=0}^\infty$ will be called a frame-generating sequence if the system $ \mathcal{X}[(\Psi^j)_{j=0}^\infty]$ forms a frame for $L^2(\mathbb{S}^{d-1})$, i.e., if there are constants $0<C_1\leq C_2$ such that 
\begin{equation}\label{def frame}
C_1 \|f \|_{L^2(\mathbb{S}^{d-1})}^2  \leq \sum_{j=0}^\infty \sum_{r=1}^{r_j} \mu_{j, r}\, \lvert \langle f, T^d(g_{j, r})\Psi^j\rangle_{\mathbb{S}^{d-1}}\rvert^2 \leq  C_2 \|f \|_{L^2(\mathbb{S}^{d-1})}^2, 
\end{equation}
holds for all $f \in L^2(\mathbb{S}^{d-1})$. For further details on frame theory, we refer to \cite{christensen2016introduction}. 

Clearly, systems of the form \eqref{ourFrames} are constructed in a similar way to classical wavelet frames, with $T^d$ taking on the role of the translation operator. Hence, in reference to the wavelet setting, we will often say that $\Psi^j$ belongs to the scale $j$.

If $\kappa_j \colon \{0, 1, \dots, N_j \} \rightarrow \mathbb{C}$, $j \in \mathbb{N}_0$, are suitable frequency filters, choosing
\begin{equation*}
    \Psi^j = \sum_{n=0}^{N_j}\kappa_{j}(n)\sum_{k \in \mathcal{I}_n^d} \overline{Y_k^{d, n}(e^d)} Y_k^{d, n}, \quad j \in \mathbb{N}_0,
\end{equation*}
yields the well-established isotropic setting studied in \cite{bib11, bib3, bib15, bib16, bib29, bib30, bib14, bib46, bib48, bib49, bib25, bib78}. Here, each initial function $\Psi^j$ is invariant under rotations that keep the North Pole $e^d$ fixed and, as a consequence, it is not hard to see that the integration rules \eqref{eq35} can be replaced by suitable spherical quadrature formulas. For more details, we refer to the discussion in \autoref{subsec4.5}.

The following proposition gives a characterization of all frames $\mathcal{X}[(\Psi^j)_{j=0}^\infty]$ in terms of an admissibility condition on the Fourier coefficients $\Psi^j(n, k) = \langle \Psi^j, Y_k^{d, n}\rangle_{\mathbb{S}^{d-1}}$.
\begin{proposition}\label{prop2}
Let $0<C_1 \leq C_2$ and $(\Psi^j)_{j=0}^\infty$ be a sequence of polynomials such that $\Psi^j \in \Pi_{N_j}(\mathbb{S}^{d-1})$ for each $j \in \mathbb{N}_0$. Then $ \mathcal{X}[(\Psi^j)_{j=0}^\infty]$ is a frame for $L^2(\mathbb{S}^{d-1})$ with lower and upper frame bounds $C_1$ and $C_2$, respectively, if and only if 
\begin{equation}\label{eq32}
C_1 \leq (\dim \mathcal{H}_n^d)^{-1}\sum_{j=0}^\infty \sum_{k \in \mathcal{I}_n^d} \lvert \Psi^j(n, k)\rvert^2 \leq C_2, \qquad \text{for each } n \in \mathbb{N}_0. 
\end{equation}

\end{proposition}
\begin{proof}
We need to show that \eqref{eq32} is equivalent to \eqref{def frame}. 
Expanding $T^d(g)\Psi^j$ in terms of the matrix functions defined by \eqref{eq28}, we have
\begin{equation}\label{eq34}
T^d(g)\Psi^j = \sum_{n=0}^{N_j} \sum_{\ell \in \mathcal{I}_n^d} \sum_{k \in  \mathcal{I}_n^d}  \Psi^j(n, k)\, t_{\ell, k}^{d, n}(g) \, Y_\ell^{d, n}.
\end{equation}
In particular, the functions $g \mapsto T^d(g)\Psi^j(\eta)$ and $g \mapsto \langle f, T^d(g)\Psi^j \rangle_{\mathbb{S}^{d-1}} $, where $\eta \in \mathbb{S}^{d-1}$ and $f \in L^1(\mathbb{S}^{d-1})$, are contained in the polynomial subspace $\mathcal{M}_{N_j}^1(SO(d))$. Consequently, it follows from \eqref{eq35} that 
\begin{align}\label{eq:d-c}
\sum_{j=0}^\infty \sum_{r=1}^{r_j} \mu_{j, r} \, \lvert \langle f, T^d(g_{j, r})\Psi^j\rangle_{\mathbb{S}^{d-1}} \rvert^2 = \sum_{j=0}^\infty \int_{SO(d)} \lvert \langle f, T^d(g)\Psi^j\rangle_{\mathbb{S}^{d-1}} \rvert^2 \, \mathrm{d}\mu_d(g). 
\end{align}
Now, similar to \eqref{eq34}, we have
\begin{equation*}
\langle f, T^d(g)\Psi^j\rangle_{\mathbb{S}^{d-1}}  = \sum_{n=0}^\infty \sum_{k \in \mathcal{I}_n^d} \sum_{\ell \in \mathcal{I}_n^d} \overline{\Psi^j(n, k)} \, \langle f, Y_\ell^{d, n} \rangle_{\mathbb{S}^{d-1}} \, \overline{t_{\ell, k}^{d, n}(g)}.
\end{equation*}
Thus, by Parseval's identity,
\begin{align}
&\sum_{j=0}^\infty \int_{SO(d)} \lvert \langle f, T^d(g)\Psi^j\rangle_{\mathbb{S}^{d-1}} \rvert^2 \, \mathrm{d}\mu_d(g)  = \nonumber\\
& \qquad \qquad  \sum_{j=0}^\infty  \sum_{n=0}^\infty \sum_{k \in \mathcal{I}_n^d} \sum_{\ell \in \mathcal{I}_n^d} (\dim \mathcal{H}_n^d)^{-1}\lvert\Psi^j(n, k) \rvert^2 \, \lvert \langle f, Y_\ell^{d, n} \rangle_{\mathbb{S}^{d-1}} \rvert^2. \label{eq:c-f}
\end{align}
Hence, if \eqref{eq32} holds, the frame property follows immediately. On the other hand, if \eqref{def frame} is satisfied, then by substituting $f=Y_{\ell}^{d,n}$ into \eqref{eq:d-c} and applying the equality \eqref{eq:c-f}, we obtain
\begin{equation*}
C_1 \leq (\dim \mathcal{H}_n^d)^{-1}\sum_{j=0}^\infty \sum_{k \in \mathcal{I}_n^d} \lvert \Psi^j(n, k)\rvert^2 \leq C_2.
\end{equation*}
\end{proof}

The following result characterizes all possible dual frames within our construction.
\begin{proposition}\label{prop dual pairs}
Let $(\Psi^j)_{j=0}^\infty$ and $(\tilde{\Psi}^j)_{j=0}^\infty$ be two frame-generating sequences with $\Psi^j$, $\tilde{\Psi}^j\in \Pi_{N_j}(\mathbb{S}^{d-1})$ for each $j \in \mathbb{N}_0$. Then $\mathcal{X}[(\Psi^j)_{j=0}^\infty]$ and $\mathcal{X}[(\tilde{\Psi}^j)_{j=0}^\infty]$ are dual frames if and only if
\begin{equation}\label{eq40}
(\dim \mathcal{H}_n^d)^{-1} \sum_{j=0}^\infty \sum_{k \in \mathcal{I}_n^d} \overline{\Psi^j(n, k)} \, \tilde{\Psi}^j(n, k) = 1, \qquad \text{for all } n \in \mathbb{N}_0.
\end{equation}
\end{proposition}
\begin{proof}
We need to show that condition \eqref{eq40} is equivalent to
\begin{equation}\label{eq39}
f = \sum_{j=0}^\infty \sum_{r=1}^{r_j}\mu_{j, r}\, \langle f, T^d(g_{j, r})\Psi^j \rangle_{\mathbb{S}^{d-1}} \, T^d(g_{j, r})\tilde{\Psi}^j, \qquad \text{for each }f \in L^2(\mathbb{S}^{d-1}).
\end{equation}
Assume that \eqref{eq39} holds true for all $f \in L^2(\mathbb{S}^{d-1})$. As in the proof of \hyperref[prop2]{Proposition~\ref*{prop2}}, we easily derive
\begin{align}\label{eq:23,prop3.2}
&\sum_{j=0}^\infty \sum_{r=1}^{r_j}\mu_{j, r}\, \langle f, T^d(g_{j, r})\Psi^j \rangle_{\mathbb{S}^{d-1}} \, T^d(g_{j, r})\tilde{\Psi}^j(\eta) \nonumber \\
&\qquad \qquad = \sum_{j=0}^\infty  \int_{SO(d)}\langle f, T^d(g)\Psi^j \rangle_{\mathbb{S}^{d-1}} \, T^d(g)\tilde{\Psi}^j(\eta)\, \mathrm{d}\mu_d(g) \nonumber \\
&\qquad \qquad   = \sum_{j=0}^\infty  \sum_{n=0}^\infty \sum_{k \in \mathcal{I}_n^d}\sum_{\ell \in \mathcal{I}_n^d} \langle f, Y_\ell^{d, n}\rangle_{\mathbb{S}^{d-1}} Y_\ell^{d, n}(\eta) \, (\dim \mathcal{H}_n^d)^{-1}  \overline{\Psi^j(n, k)} \, \tilde{\Psi}^j(n, k).
\end{align}
Now, setting $f=Y_\ell^{d, n}$, we immediately get
\begin{equation*}
(\dim \mathcal{H}_n^d)^{-1} \sum_{j=0}^\infty \sum_{k \in \mathcal{I}_n^d} \overline{\Psi^j(n, k)} \, \tilde{\Psi}^j(n, k) = 1.
\end{equation*}
 In the other direction, since $(\Psi^j)_{j=0}^\infty$ and $(\tilde{\Psi}^j)_{j=0}^\infty$ are frame-generating sequences, the infinite series in \eqref{eq39} is convergent for every $f \in L^2(\mathbb{S}^{d-1})$ by Hölder's inequality. Then taking the inner product of both sides of  \eqref{eq:23,prop3.2} with $Y_{\ell'}^{d,n'}$ yields
 \begin{align*}
 &\sum_{j=0}^\infty \sum_{r=1}^{r_j}\mu_{j, r}\, \langle f, T^d(g_{j, r})\Psi^j \rangle_{\mathbb{S}^{d-1}} \, 
 \langle T^d(g_{j, r})\tilde{\Psi}^j, Y_{\ell'}^{d,n'}\rangle_{\mathbb{S}^{d-1}} \\
 &= \sum_{j=0}^\infty  \sum_{n=0}^\infty \sum_{k \in \mathcal{I}_n^d}\sum_{\ell \in \mathcal{I}_n^d} \langle f, Y_\ell^{d, n}\rangle_{\mathbb{S}^{d-1}}
 \langle Y_\ell^{d, n}, Y_{\ell'}^{d, n'}\rangle_{\mathbb{S}^{d-1}} \, (\dim \mathcal{H}_n^d)^{-1}  \overline{\Psi^j(n, k)} \, \tilde{\Psi}^j(n, k) \\
 &= \sum_{j=0}^\infty   \sum_{k \in \mathcal{I}_n^d}\langle f,   Y_{\ell'}^{d, n'}\rangle_{\mathbb{S}^{d-1}} \, (\dim \mathcal{H}_{n'}^d)^{-1}  \overline{\Psi^j(n', k)} \, \tilde{\Psi}^j(n', k)\\
 &= \langle f,   Y_{\ell'}^{d, n'}\rangle_{\mathbb{S}^{d-1}}
 \end{align*}
 where the last equality follows from \eqref{eq40}. Since this equality holds for every $n'$ and $\ell'$, it follows that 
 \[
f = \sum_{j=0}^\infty \sum_{r=1}^{r_j}\mu_{j, r}\, \langle f, T^d(g_{j, r})\Psi^j \rangle_{\mathbb{S}^{d-1}} \, T^d(g_{j, r})\tilde{\Psi}^j. 
 \]
This completes the proof.
\end{proof}

Let $(\Psi^j)_{j=0}^\infty$, $(\tilde{\Psi}^j)_{j=0}^\infty$ be two frame-generating sequences with $\Psi^j, \tilde{\Psi}^j \in \Pi_{N_j}(\mathbb{S}^{d-1})$ that form a dual pair. Then
\begin{equation*}
  f = \sum_{j=0}^\infty \sum_{r=1}^{r_j}\mu_{j, r}\, \langle f, T^d(g_{j, r})\Psi^j \rangle_{\mathbb{S}^{d-1}} \, T^d(g_{j, r})\tilde{\Psi}^j, \quad \text{for each } f \in L^2(\mathbb{S}^{d-1}),
\end{equation*}
and the series on the right-hand side converges unconditionally w.r.t.\ $\| \cdot \|_{L^2(\mathbb{S}^{d-1})}$. Now, we consider the sequence of approximation operators
\begin{equation*}
    \Lambda_J \colon L^1(\mathbb{S}^{d-1}) \rightarrow \Pi_{N_J}(\mathbb{S}^{d-1}), \quad \Lambda_J f = \sum_{j=0}^J \sum_{r=1}^{r_j}\mu_{j, r}\, \langle f, T^d(g_{j, r})\Psi^j \rangle_{\mathbb{S}^{d-1}} \, T^d(g_{j, r})\tilde{\Psi}^j.
\end{equation*}
As in \eqref{eq:23,prop3.2}, a straightforward calculation shows that
\begin{equation}\label{approximation operators}
   \Lambda_J f = \sum_{n=0}^{N_J} \sum_{\ell \in \mathcal{I}_n^d}\sigma_J(n) \langle f, Y_\ell^{d, n}\rangle_{\mathbb{S}^{d-1}} Y_\ell^{d, n},
\end{equation}
where
\begin{equation*}
    \sigma_J(n) = (\dim \mathcal{H}_n^d)^{-1}\sum_{j=0}^J \sum_{k \in \mathcal{I}_n^d}  \overline{\Psi^j(n, k)} \, \tilde{\Psi}^j(n, k).
\end{equation*}
Operators of the form \eqref{approximation operators} are well investigated in the literature (see, e.g., \cite{bib3} for a detailed review). In particular, if
\begin{equation}\label{sigma_J eq}
    \sigma_J(n) = 1, \quad \text{for } 0 \leq n \leq N_{j-1}
\end{equation}
and if the operators $ \Lambda_J$, $J \in \mathbb{N}_0$ are uniformly bounded in $L^p(\mathbb{S}^{d-1})$, $1 \leq p < \infty$, i.e.,
\begin{equation}\label{uniformly bounded L^p}
    \| \Lambda_J f \|_{L^p(\mathbb{S}^{d-1})} \leq c_p \| f\|_{L^p(\mathbb{S}^{d-1})}, \quad \text{for each } f \in L^p(\mathbb{S}^{d-1}),
\end{equation}
then a standard argument yields
\begin{equation*}
    \| f- \Lambda_Jf \|_{L^p(\mathbb{S}^{d-1})} \leq c_p E_{N_{J-1}}(f)_p, \quad \text{for each } f \in L^p(\mathbb{S}^{d-1}).
\end{equation*}
Here,
\begin{equation*}
    E_m(f)_p = \inf_{P \in \Pi_m(\mathbb{S}^{d-1})} \| f-P \|_{L^p(\mathbb{S}^{d-1})}
\end{equation*}
denotes the error of the best approximation of $f \in L^p(\mathbb{S}^{d-1})$ in $\Pi_{m}(\mathbb{S}^{d-1})$ w.r.t.\ $\| \cdot \|_{L^p(\mathbb{S}^{d-1})}$. Analogously, if \eqref{sigma_J eq} holds and if
\begin{equation}\label{uniformly bounded C}
    \| \Lambda_J f \|_{C(\mathbb{S}^{d-1})} \leq c \| f\|_{C(\mathbb{S}^{d-1})}, \quad \text{for each } f \in C(\mathbb{S}^{d-1}),
\end{equation}
then
\begin{equation*}
    \| f- \Lambda_Jf \|_{C(\mathbb{S}^{d-1})} \leq c \inf_{P \in \Pi_{N_{J-1}}(\mathbb{S}^{d-1})} \| f-P \|_{C(\mathbb{S}^{d-1})}, \quad \text{for each } f \in C(\mathbb{S}^{d-1}).
\end{equation*}

A standard construction that satisfies \eqref{sigma_J eq}, \eqref{uniformly bounded L^p} and \eqref{uniformly bounded C} can be given as follows: Let $\phi \in C^\infty([0, \infty))$ be a real-valued non-increasing function such that $\phi(t)=1$ for $t\in [0, 1/2]$ and $\phi(t)=0$ for $t\geq 1$. A corresponding window function with support in $[1/2, 2]$ is defined by
\begin{equation*}
\kappa(t)=\sqrt{\phi^2(t/2) - \phi^2(t)}, \qquad 0 \leq t < \infty.
\end{equation*}
If $\sigma_0(n) = \delta_{n,0}$ and
\begin{equation*}
    (\dim \mathcal{H}_n^d)^{-1} \sum_{k \in \mathcal{I}_n^d}  \overline{\Psi^j(n, k)} \, \tilde{\Psi}^j(n, k) = \kappa^2 \mleft( \frac{n}{2^{j-1}} \mright), \quad n \in \mathbb{N}, \; j \in \mathbb{N},
\end{equation*}
then it follows that
\begin{equation*}
    \Lambda_J f = \sum_{n=0}^\infty \sum_{k \in \mathcal{I}_n^d}   \phi^2\!\left( \frac{n}{2^J}\right)  \langle f, Y_k^{d, n} \rangle Y_k^{d, n}.
\end{equation*}
Now, it is obvious that \eqref{sigma_J eq} holds for $N_0=0$, $N_j = 2^j$, $j \in \mathbb{N}$. Furthermore, it is well known that, in this setting, \eqref{uniformly bounded L^p} and \eqref{uniformly bounded C} are satisfied. For a proof of the latter, we refer to \cite[Theorem~2.6.3]{bib3}.

The following result shows that the canonical dual frame not only shares the same structural form as the original system, but that its generating sequence is obtained by a simple rescaling of the Fourier coefficients of the initial sequence.

\begin{proposition}\label{prop:canonical_dual}
Let $(\Psi^j)_{j=0}^\infty$ be a frame-generating sequence with $\Psi^j \in \Pi_{N_j}(\mathbb{S}^{d-1})$ for each $j \in \mathbb{N}_0$. Define $(\tilde{\Psi}^j)_{j=0}^\infty \subset \Pi(\mathbb{S}^{d-1})$ by
\begin{equation*}\label{dual sequence}
\tilde{\Psi}^j(n, k) = \sigma_n^{-1} \Psi^j(n, k), \quad\text{ where} \quad \sigma_n=(\dim \mathcal{H}_n^d)^{-1}\sum_{j=0}^\infty \sum_{k \in \mathcal{I}_n^d} \lvert \Psi^j(n, k)\rvert^2. 
\end{equation*}
Then the system $\mathcal{X}[(\tilde{\Psi}^j)_{j=0}^\infty]$ is the canonical dual frame of $\mathcal{X}[(\Psi^j)_{j=0}^\infty]$.
\end{proposition}
\begin{proof}
Let $S: L^{2}({\mathbb{ S }}^{d-1}) \to L^{2}({\mathbb{ S }}^{d-1})$ denote the frame operator corresponding to the frame $\mathcal{X}[(\Psi^j)_{j=0}^\infty]$. For any $n,\ell$, it holds that
\begin{align}\label{eq:S}
S Y_{\ell}^{d,n} &= 
\sum_{j=0}^\infty  \sum_{r=1}^{r_j}  \mu_{j, r} \langle Y_\ell^{d,n} , T^d(g_{j, r})\Psi^j \rangle_{\mathbb{S}^{d-1}}\,  T^d(g_{j, r})\Psi^j \nonumber\\
&= \sum_{j=0}^\infty \int_{SO(d) } \langle Y_{\ell}^{d,n}, T^d(g) \Psi^{j}\rangle_{\mathbb{S}^{d-1}} T^d(g) \Psi^{j} \, \mathrm{d}\mu_d(g).
\end{align}
Via \eqref{eq34}, we have 
\begin{equation}\label{eq:inner-prod}
\langle Y_{\ell}^{d, n}, T^d(g) \Psi^{j}\rangle_{\mathbb{S}^{d-1}}  
=\sum_{k \in \mathcal{I}_n^d} \overline{\Psi^{j}(n, k ) \, t_{\ell, k}^{d,n}(g)}.
\end{equation}
Substituting \eqref{eq:inner-prod} in \eqref{eq:S} and expanding $T^d(g)\Psi^j$ using the  matrix functions, we obtain 
\begin{align*}
S Y_{\ell}^{d, n} & =\sum_{j=0}^\infty   \int_{SO(d) } \sum_{k \in \mathcal{I}_n^d} \overline{\Psi^{j}(n, k) t_{\ell, k}^{d, n}(g) }
\bigg(\sum_{n'=0}^\infty \sum_{\ell', k' \in \mathcal{I}_{n'}^d}   \Psi^j(n', k')\, t_{\ell', k'}^{d, n'}(g) \, Y_{\ell'}^{d, n'}\bigg) \mathrm{d}\mu_d(g) \\
& = (\dim \mathcal{H}_n^{d}	)^{-1} \sum_{j=0}^\infty  \sum_{k \in \mathcal{I}_n^d}  \lvert\Psi^{j}(n, k )\rvert^{2} Y_{\ell}^{d,n}.
\end{align*}
Thus we have
\begin{equation*}
S Y_{\ell}^{d,n}=\sigma_{n} Y_{\ell}^{d,n}, \quad \text { where } \quad \sigma_{n}= (\dim \mathcal{H}_n^{d} )^{-1}	\sum_{j=0}^\infty  \sum_{k \in \mathcal{I}_n^d}  \lvert\Psi^{j}\left(n, k\right)\rvert^{2}.
\end{equation*}
Since $S$ is diagonal with respect to the basis $Y_{\ell}^{d,n}$, equation \eqref{eq30} implies that  $T^d(g)$ commutes with the frame operator $S$, and therefore also with $S^{-1}$, for all $g\in SO(d)$. Thus,
\begin{align*}
\sqrt{\mu_{j, r}}\, S^{-1} T^d(g_{j,r}) \Psi^{j} &=\sqrt{\mu_{j, r}}\,T^d(g_{j, r}) S^{-1} \Psi^{j}
= \sqrt{\mu_{j, r}}\,T^d(g_{j, r})\tilde{ \Psi}^{j}
\end{align*}
and the canonical dual frame associated with $\mathcal{X}[(\Psi^j)_{j=0}^\infty]$ is given by $\mathcal{X}[(\tilde{\Psi}^j)_{j=0}^\infty]$.
\end{proof}

\section{Properties}\label{sec4}
The class of polynomial frames introduced in \autoref{sec3} contains a variety of systems with different structures and characteristics. In this section, we will highlight several particularly relevant properties that a frame $\mathcal{X}[(\Psi^j)_{j=0}^\infty]$ can exhibit. Moreover, we demonstrate how these properties can be enforced during the construction process by setting suitable conditions on the Fourier coefficients $\Psi^j(n, k)$. To conclude this section, we provide a detailed discussion on various quadrature rules that can be used to build frames of the form $\mathcal{X}[(\Psi^j)_{j=0}^\infty]$ as defined in \eqref{ourFrames}.

\subsection{Steerability}\label{subsec4.1}
In the context of wavelets on the two-dimensional sphere $\mathbb{S}^2$, the notion of steerability was introduced in \cite{bib64} and later utilized in the construction of directional polynomial  wavelets \cite{bib9, bib10, bib21}. More recently, the concept has also been applied to higher-dimensional settings \cite{bib37}. A general definition, valid for all spheres $\mathbb{S}^{d-1}$, $d\geq 3$, can be given as follows.
\begin{definition}\label{steerability}
Let $K\in \mathbb{N}_0$. A sequence $(\Psi^j)_{j=0}^\infty \subset \Pi(\mathbb{S}^{d-1})$ is called $K$-steerable w.r.t.\ $SO(d-1)$, or simply steerable, if there exist $h_1, \dots, h_M \in SO(d-1)$ and $v_1, \dots, v_M \in \mathcal{M}_K^1(SO(d-1))$ such that
\begin{equation*}
T^d(h)\Psi^j = \sum_{p=1}^M v_p(h) \, T^d(h_p)\Psi^j, \qquad \text{for each }  h \in SO(d-1), \; j \in \mathbb{N}_0.
\end{equation*}
A system $\mathcal{X}[(\Psi^j)_{j=0}^\infty]$ generated by a $K$-steerable sequence $(\Psi^j)_{j=0}^\infty$ will also be called $K$-steerable (or simply steerable).
\end{definition}

Steerability is a property of particular relevance in the context of localized frames. 
Let  $\mathcal{X}[(\Psi^j)_{j=0}^\infty]$ be  a localized frame generated by a sequence $(\Psi^j)_{j=0}^\infty \subset \Pi(\mathbb{S}^{d-1})$, where each $\Psi^j$ is concentrated at the North Pole $e^d$. In this case, the functions $T^d(h)\Psi^j$, $h \in SO(d-1)$, remain concentrated at the North Pole; we refer to them as orientations of $\Psi^j$. By definition, if the sequence of polynomials $\Psi^j$ is steerable, then every orientation $T^d(h)\Psi^j$ can be expressed as a linear combination of a finite set of predetermined orientations $T^d(h_p)\Psi^j$, $p=1, \dots, M$. Consequently, for each position $\eta \in \mathbb{S}^{d-1}$, the frame coefficients $\langle f, T^d(g_\eta h)\Psi^j\rangle_{\mathbb{S}^{d-1}}$, $h \in SO(d-1)$, are completely determined by the finite set of samples
\begin{equation}\label{frame coeff eq1}
    \langle f, T^d(g_\eta h_p)\Psi^j\rangle_{\mathbb{S}^{d-1}}, \quad p=1, \dots, M.
\end{equation}
In this sense, steerable frames can also be interpreted as frames with limited directional resolution (or sensitivity), as they probe only a fixed number of orientations at each position. The degree of directionality associated with a $K$-steerable sequence grows with $K$. In particular, zonal/isotropic frames correspond precisely to the case $K=0$, for which one may take $M=1$ in \eqref{frame coeff eq1}.

Although limiting directional sensitivity can be a disadvantage in some situations, steerability also comes with considerable benefits. Indeed, in contrast to the general case, steerable frames require fewer rotations at each scale, leading to a reduction in computational cost. For a more detailed discussion of the latter, we refer to \autoref{subsec4.5}. Moreover, in the articles \cite{bib9,bib31, bib37} the property of steerability was essential in deriving a quasi-exponential localization bound for directional wavelets. Indeed, for our investigations in \autoref{subsec4.3}, where we study frames that are optimally localized in terms of a spherical uncertainty principle, steerability will also be required.

The following proposition provides a useful characterization of steerable systems in terms of their analysis coefficients.
\begin{proposition}\label{propSteerable}
A sequence $(\Psi^j)_{j=0}^\infty \subset \Pi(\mathbb{S}^{d-1})$ is $K$-steerable w.r.t.\ $SO(d-1)$ if and only if, for each $j \in \mathbb{N}_0$ and $f \in L^2(\mathbb{S}^{d-1})$, the function
\begin{equation*}
h \mapsto \langle f, T^d(h)\Psi^j\rangle_{\mathbb{S}^{d-1}}
\end{equation*}
belongs to $\mathcal{M}_K^1(SO(d-1))$.
\end{proposition}
\begin{proof}
If $(\Psi^j)_{j=0}^\infty$ is steerable, it follows directly that the map $h \mapsto \langle f, T^d(h)\Psi^j\rangle_{\mathbb{S}^{d-1}}$ is in $\mathcal{M}_K^1(SO(d-1))$.

For the other direction, we assume that
\begin{equation*}
h \mapsto \langle f, T^d(h)\Psi^j\rangle_{\mathbb{S}^{d-1}}
\end{equation*}
belongs to $\mathcal{M}_K^1(SO(d-1))$ for any $j \in \mathbb{N}_0$ and $f \in L^2(\mathbb{S}^{d-1})$. Let $h_p \in SO(d-1)$, $w_p\in \mathbb{C}$ for $p=1, \dots, M$ such that
\begin{equation*}
\sum_{p=1}^M w_p \, f_1(h_p) f_2(h_p) = \int_{SO(d-1)} f_1 f_2 \, \mathrm{d}\mu_{d-1} \qquad \text{if } f_1, f_2 \in \mathcal{M}_K^1(SO(d-1)).
\end{equation*} 
Then it is easy to see that
\begin{equation*}
\langle f, T^d(h)\Psi^j\rangle_{\mathbb{S}^{d-1}} = \langle f,  \sum_{p=1}^M \overline{v_p(h)} \, T^d(h_p)\Psi^j\rangle_{\mathbb{S}^{d-1}},
\end{equation*}
where
\begin{equation*}
v_p = w_p\sum_{n=0}^K \dim \mathcal{H}_n^{d-1} \sum_{k, k' \in \mathcal{I}_n^{d-1}} \overline{t_{k, k'}^{d-1, n}(h_p)} \, t_{k, k'}^{d-1, n}.
\end{equation*}
Since the latter holds for any $f \in L^2(\mathbb{S}^{d-1})$, it follows that $(\Psi^j)_{j=0}^\infty$ is $K$-steerable.
\end{proof}

Finally, we mention that in terms of the specific choice of spherical harmonics $Y_k^{d, n}$ given in \hyperref[appendix A]{Appendix A}, the construction of steerable systems is particularly convenient.
\begin{lemma}\label{lemmasteerable}
A sequence $(\Psi^j)_{j=0}^\infty \subset \Pi(\mathbb{S}^{d-1})$ is $K$-steerable with respect to $SO(d-1)$ if and only if for every $n\in\mathbb{N}_0$ and $k=(k_1,\dots, k_{d-2}) \in \mathcal{I}_n^d$, 
\[
 \langle \Psi^j , Y_k^{d,n} \rangle_{\mathbb{S}^{d-1}} = 0, \quad \text{whenever } \lvert k_1 \rvert > K,
\]
where $Y_k^{d,n}$ is defined by \eqref{y_n^k def}.
\end{lemma}
\begin{proof}
For $d=3$, this fact was discussed in \cite{bib21}. In the case $d\geq 4$, letting $\Psi^j(n, k) = \langle \Psi^j , Y_k^{d,n} \rangle_{\mathbb{S}^{d-1}}$, a simple calculation yields
\begin{equation*}
     \langle f, T^d(h)\Psi^j\rangle_{\mathbb{S}^{d-1}} = \sum_{n=0}^\infty \sum_{k_1=0}^n \sum_{m \in \mathcal{I}_{k_1}^{d-1}}\sum_{\ell \in \mathcal{I}_{k_1}^{d-1}} \overline{\Psi^j(n, (k_1, m))} \, \overline{t_{m, \ell}^{d-1, k_1}(h)} \, \langle f, Y_{(k_1, \ell)}^{d, n}\rangle_{\mathbb{S}^{d-1}},
\end{equation*}
for each $f \in L^2(\mathbb{S}^{d-1})$ and $h\in SO(d-1)$. Thus, the assertion follows from \hyperref[propSteerable]{Proposition~\ref*{propSteerable}}.
\end{proof}

\subsection{Rotational Symmetries}\label{subsec4.2}
In the construction of polynomial frames $\mathcal{X}[(\Psi^j)_{j=0}^\infty]$ as in \eqref{ourFrames}, the incorporation of certain rotational symmetries can be another feasible strategy to reduce computational cost. To provide a formal framework, we introduce the concept of $SO(m)$-invariant systems, which have previously been utilized in the setting of highly localized and directional tight frames for $L^2(\mathbb{S}^{d-1})$ in \cite{bib37}.
\begin{definition}
    Let $m\in \{2, \dots, d-1\}$. A sequence $(\Psi^j)_{j=0}^\infty \subset \Pi(\mathbb{S}^{d-1})$ is called $SO(m)$-invariant, if
    \begin{equation*}
        T^d(h)\Psi^j = \Psi^j, \qquad \text{for each } h\in SO(m), \; j \in \mathbb{N}_0.
    \end{equation*}
    A system $\mathcal{X}[(\Psi^j)_{j=0}^\infty]$ generated by such a sequence will also be called $SO(m)$-invariant.
\end{definition}
It should be noted that an element $\sqrt{\mu_{j, r}}\, T^d(g_{j, r} ) \Psi^j$ of an $SO(m)$-invariant system $\mathcal{X}[(\Psi^j)_{j=0}^\infty]$ as defined in \eqref{ourFrames} is in general not invariant w.r.t.\ $SO(m)$ but w.r.t.\ its conjugate subgroup $g_{j, r} SO(m) g_{j, r}^{-1}$, since
\begin{equation*}
   T^d(g_{j, r} h g_{j, r}^{-1}) T^d(g_{j, r} ) \Psi^j = T^d(g_{j, r} ) \Psi^j, \qquad \text{for each } h \in SO(m).
\end{equation*}

Compared to the general case, $SO(m)$-invariant frames $\mathcal{X}[(\Psi^j)_{j=0}^\infty]$ require fewer rotations at each scale $j$ (see \autoref{subsec4.5} for more details). In this context, $SO(d-1)$-invariant systems, consisting of zonal (isotropic) functions, are the most computationally efficient, yet they provide no directional sensitivity. Thus, $SO(d-2)$-invariant systems are of particular interest, as they allow for highly directional constructions while retaining a substantial degree of symmetry. Explicit examples of such frames are presented in \autoref{sec5}.

As with steerability, $SO(m)$-invariance admits a particularly simple representation in terms of the spherical harmonics defined in \eqref{y_n^k def}. Indeed, it is straightforward to see that a sequence $(\Psi^j)_{j=0}^\infty \subset \Pi(\mathbb{S}^{d-1})$ is $SO(m)$-invariant if and only if 
\begin{equation}\label{k_(d-m)>0}
        \Psi^j(n, k)=0, \quad \text{whenever } k_{d-m}\neq0.
\end{equation}

\subsection{Localization}\label{subsec4.3}
While localization properties of zonal/isotropic polynomial frames have been extensively studied in the literature (see \cite{bib40, bib48, bib49, bib11, bib14, bib15, bib30, bib3, bib16, bib25, bib78}), comparatively little is known in the directional setting.  Here, pointwise localization estimates have been established for certain systems on $\mathbb{S}^2$ in \cite{bib26, bib31, bib9} and on $\mathbb{S}^{d-1}$ with $d \geq 3$ in \cite{bib37}. As in the zonal case, such estimates require a high order of smoothness in the Fourier domain.

In this subsection, we investigate localization in terms of a well-known spherical uncertainty principle \cite{bib38}. Specifically, we derive sufficient conditions under which a frame $\mathcal{X}[(\Psi^j)_{j=0}^\infty]$ is optimally concentrated in space, thereby generalizing a previous result for zonal systems given in \cite{bib11}. The conditions imposed on the Fourier coefficients $\Psi^j(n, k)$ are relatively mild compared to those required for pointwise estimates, allowing for a broad class of well-localized non-zonal polynomial frames.

Many of the following considerations and calculations are similar to those in \cite{bib11, bib40, bib41}, and we adopt some of the notations used therein. All constants $c, c_1, c_2, \dots$ in this subsection are understood to be independent of the scale $j \in \mathbb{N}_0$, although they may depend on $d$. Moreover, the values of the constants can change with each appearance. For two nonnegative sequences $(a_j)_{j=0}^\infty$ and $(b_j)_{j=0}^\infty$, we write $a_j\sim b_j$ to indicate that there exist constants $0<c_1\leq c_2$ such that
\begin{equation*}
    c_1 a_j \leq b_j \leq c_2 a_j, \qquad \text{for each } j \in \mathbb{N}_0.
\end{equation*}

The spherical uncertainty principle given in \cite[Corollary~5.1]{bib38} implies that for any $f \in C^2(\mathbb{S}^{d-1})$ the inequality
\begin{equation}\label{goh_goodman}
\frac{(d-1)^2}{4} \leq \left( \frac{\| f \|_{L^2(\mathbb{S}^{d-1})}^4}{\| \int_{\mathbb{S}^{d-1}} x \lvert f(x)\rvert^2 \, \mathrm{d}\omega_{d-1}(x) \|_2^2  }-1 \right) \frac{- \langle \Delta_{\mathbb{S}^{d-1}} f , f \rangle_{\mathbb{S}^{d-1}}}{\| f \|_{L^2(\mathbb{S}^{d-1})}^2}
\end{equation}
holds, provided that the denominators on the right-hand side are all nonzero. Here, we note that the original formula in \cite[Corollary~5.1]{bib38} is given in terms of the surface gradient, instead of the Laplace-Beltrami operator. To obtain \eqref{goh_goodman}, we utilized the classical Green-Beltrami identity (see, e.g., \cite[Proposition~3.3]{bib75}).  We also note that the uncertainty principle \eqref{goh_goodman} is a generalization of earlier results in \cite{bib42} and \cite{bib45}, where the corresponding inequality was established for signals on the $2$-sphere and zonal functions on $\mathbb{S}^{d-1}$, $d\geq 3$, respectively. Using the notations
\begin{equation}\label{VarS}
\mathrm{Var}_{\mathrm{S}}(f) = \frac{1-\| \xi_0(f) \|_2^2}{\| \xi_0(f) \|_2^2}, \qquad \xi_0(f) = \frac{ \int_{\mathbb{S}^{d-1}} x \lvert f(x) \rvert^2 \, \mathrm{d}\omega_{d-1}(x) }{\| f\|_{L^2(\mathbb{S}^{d-1})}^2},
\end{equation}
and
\begin{equation*}
\mathrm{Var}_{\mathrm{M}}(f) = \frac{- \langle \Delta_{\mathbb{S}^{d-1}} f , f \rangle_{\mathbb{S}^{d-1}}}{\| f \|_{L^2(\mathbb{S}^{d-1})}^2},
\end{equation*}
for the variance in space and momentum, respectively, \eqref{goh_goodman} can be written as
\begin{equation}\label{goh_goodman2}
\frac{(d-1)^2}{4} \leq \mathrm{Var}_{\mathrm{S}}(f) \,  \mathrm{Var}_{\mathrm{M}}(f).
\end{equation}

We are interested in polynomial frames $\mathcal{X}[(\Psi^j)_{j=0}^\infty]$, as given in \eqref{ourFrames}, that are optimally localized in space in terms of the above uncertainty principle. Clearly, finding such a system reduces to finding a frame-generating sequence $(\Psi^j)_{j=0}^\infty\subset \Pi(\mathbb{S}^{d-1})$, such that $(\mathrm{Var}_{\mathrm{S}}(\Psi^j))_{j=0}^\infty$ decays at an optimal rate. Assuming that $\Psi^j\in \Pi_{N_j}(\mathbb{S}^{d-1})$, it is easy to see that
\begin{equation*}
 \mathrm{Var}_{\mathrm{M}}(\Psi^j) = \frac{\sum_{n=0}^{N_j}n(n+d-2) \sum_{k \in \mathcal{I}_n^d} \lvert \Psi^j(n, k) \rvert^2}{\| \Psi^j\|_{L^2(\mathbb{S}^{d-1})}^2} \leq c N_j^2.
\end{equation*}
Consequently, the uncertainty principle \eqref{goh_goodman2} implies that the spatial localization is limited by
\begin{equation*}
\mathrm{Var}_{\mathrm{S}}(\Psi^j) \geq c N_j^{-2}.
\end{equation*}
Therefore, optimal localization is obtained if $\mathrm{Var}_{\mathrm{S}}(\Psi^j) \leq c N_j^{-2}$, i.e., if 
\begin{equation}\label{eqoptimal}
\mathrm{Var}_{\mathrm{S}}(\Psi^j) \sim N_j^{-2}.
\end{equation}
Moreover, if 
\begin{equation*}
    1-\xi_0^d(\Psi^j) \leq cN_j^{-2},
\end{equation*}
where $\xi_0^d (\Psi^j) = \langle \xi_0(\Psi^j), e^d \rangle$, then it is easy to see that the corresponding sequence $(\Psi^j)_{j=0}^\infty$ is optimally localized (in the sense of \eqref{eqoptimal}) at the North Pole $e^d$.

We note that the polynomials of minimal space variance w.r.t.\ \eqref{goh_goodman2} have been computed in \cite[Theorem~4.3]{bib11} and \cite{bib40}, and they turn out to be zonal/isotropic. However, as in \cite[Theorem~4.1]{bib11}, we are merely interested in sequences of functions whose localization rate is of the optimal order in the sense of \eqref{eqoptimal}.

In order to compute the space variance $\mathrm{Var}_{\mathrm{S}}(\Psi^j)$, we will work with a well-known explicit orthonormal basis of spherical harmonics (see, e.g., \cite[p.~466]{bib32}) presented in \hyperref[appendix A]{Appendix A}. Thus, for the remainder of this subsection (and, in fact, the remainder of this article) we will exclusively refer to the specific basis of spherical harmonics in \eqref{y_n^k def} as well as the corresponding index set \eqref{indexset} when using the symbols $Y_k^{d, n}$ and $\mathcal{I}_n^d$, respectively. With respect to this explicitly given system, we can now compute the space variance of a signal $f$ in terms of its Fourier coefficients $f(n, k)=\langle f , Y_k^{d, n}\rangle_{\mathbb{S}^{d-1}}$, as previously demonstrated in \cite{bib41}. In particular, we will make use of the following lemma. Here, our convention that $f(n, k)=0$ if $k \notin \mathcal{I}_n$ or $n=-1$ proves to be useful.
\begin{lemma}\label{lemmaUP}
The $d$-th component of the vector $\xi_0(f) = (\xi_0^1(f), \dots, \xi_0^d(f))$ defined in \eqref{VarS} satisfies
\begin{equation}\label{eqUP}
\xi_0^d(f) \|f \|_{L^2(\mathbb{S}^{d-1})}^2 = \sum_{n=0}^\infty \sum_{k \in \mathcal{I}_n^d} f(n, k) \left( \overline{f(n+1, k)} \, Q_{d}^{k_1}(n) + \overline{f(n-1, k)} \,Q_{d}^{k_1}(n-1) \right),
\end{equation}
where
\begin{equation*}
Q_d^{k_1}(n) = \frac{1}{2}\sqrt{1 - \frac{q_d(k_1)}{n^2+n(d-1)+d(d-2)/4}},
\end{equation*}
and
\begin{equation*}
q_d(k_1)=\lvert k_1 \rvert^2 + \lvert k_1 \rvert (d-3) - (d-2)(1-d/4).
\end{equation*}
\end{lemma}
The above representation of $\xi_0^d(f)$ follows from a more general result in \cite{bib41} and has been established for zonal signals in \cite{bib11}. However, we were unable to find an exact reference from which the formula \eqref{eqUP} is readily apparent. Hence, we decided to provide a proof of the latter in \hyperref[appendix B]{Appendix B}.

In order to prove an optimal localization rate of the form \eqref{eqoptimal}, we will require a number of conditions on the sequence $(\Psi^j)_{j=0}^\infty$ 
with $\Psi^j \in \Pi_{N_j}(\mathbb{S}^{d-1})$ that are listed below. 

\begin{enumerate}[label=(C\arabic*)]
  \item \label{C1} It holds that $\|\Psi^j\|_{L^2(\mathbb{S}^{d-1})}^2 \sim N_j^{d-1}$.
    \item \label{C2} There is a sequence $(M_j)_{j=0}^\infty
    \subset \mathbb{N}_0$ with $M_j \sim N_j$ such that 
      \begin{equation*}
      \Psi^j(n, k) = 0, \quad\text{ if } n \notin \{ M_j, M_j+1, \dots, N_j\}.
      \end{equation*}
  \item\label{C3} With the same sequence 
$M_j$ as in \ref{C2}, it holds for all $M_j \leq n \leq N_j$ that 
    \begin{equation*}
        \left \vert \frac{\Psi^j(n+1, k) + \Psi^j(n-1, k)}{2} - \Psi^j(n, k) \right\vert \leq cN_j^{(d-6)/2}, \phantom{ \lvert\Psi^j(n, k)\rvert}
    \end{equation*}
    where the constant $c>0$ does not depend on $n, k$ or $j$.
    
\end{enumerate}
Examples of frames satisfying these conditions are presented in \autoref{sec5}.
We are now in a position to present the main result of this section, which establishes sufficient conditions under which a polynomial sequence $(\Psi^j)_{j=0}^\infty$ is optimally localized at the North Pole $e^d$ in terms of the uncertainty principle \eqref{goh_goodman2}.

\begin{proposition}\label{proplocalization}
Let $(\Psi^j)_{j=0}^\infty $
be a $K$-steerable frame-generating sequence with $\Psi^j \in \Pi_{N_j}(\mathbb{S}^{d-1})$. If conditions \ref{C1}-\ref{C3} are satisfied, then
\begin{equation}\label{localization eq1}
    1-\xi_0^d(\Psi^j) \leq cN_j^{-2},
\end{equation}
and, consequently,
\begin{equation}\label{localization eq2}
\mathrm{Var}_{\mathrm{S}}(\Psi^j) \leq c N_j^{-2}.
\end{equation}
\end{proposition}

\begin{proof}
According to \ref{C2} and \hyperref[lemmaUP]{Lemma~\ref*{lemmaUP}}, it holds that
\begin{align}\label{eqzeta_0}
\xi_0^d(\Psi^j) \|\Psi^j\|_{L^2(\mathbb{S}^{d-1})}^2  & = \sum_{n=M_j}^{N_j} \sum_{k \in \mathcal{I}_n^d} \Psi^j(n, k) \left( \overline{\Psi^j(n+1, k)} + \overline{\Psi^j(n-1, k)}\right) \, Q_{d}^{k_1}(n) \nonumber \\
 & \qquad \qquad \qquad + \Psi^j(n, k)\overline{\Psi^j(n-1, k)} \left(  Q_{d}^{k_1}(n-1) -  Q_{d}^{k_1}(n) \right).
\end{align}
By assumption, the sequence $(\Psi^j)_{j=0}^\infty$ is $K$-steerable for some $K\in \mathbb{N}_0$. Hence, as discussed in \hyperref[lemmasteerable]{Lemma~\ref*{lemmasteerable}}, every non-zero summand in \eqref{eqzeta_0} satisfies $\lvert k_1\rvert \leq K$ and it is easy to verify that in this case
\begin{equation}\label{eqasympt1}
    Q_{d}^{k_1}(n) = \frac{1}{2} (1+\mathcal{O}(n^{-2}))
\end{equation}
as well as
\begin{equation}\label{eqasympt2}
    Q_{d}^{k_1}(n-1) -  Q_{d}^{k_1}(n) = \mathcal{O}(n^{-3})
\end{equation}
as $n \rightarrow \infty$. Here, and in the remainder of the proof, the constants given implicitly in the $\mathcal{O}$-notation only depend on $K$, $d$ and the sequence $(\Psi^j)_{j=0}^\infty$. Using \eqref{eqasympt1} and \eqref{eqasympt2}, it follows from \ref{C2} and the Cauchy-Schwarz inequality that
\begin{align}\label{eqasympt3}
    &\xi_0^d(\Psi^j) \|\Psi^j\|_{L^2(\mathbb{S}^{d-1})}^2 \nonumber \\
&  \qquad  \qquad = \mathcal{O}(N_j^{-2}) \|\Psi^j\|_{L^2(\mathbb{S}^{d-1})}^2 + \sum_{n=M_j}^{N_j} \sum_{k \in \mathcal{I}_n^d} \Psi^j(n, k) \frac{\overline{\Psi^j(n+1, k)} + \overline{\Psi^j(n-1, k)}}{2}.
\end{align}
Assuming that \ref{C3} holds, then \eqref{eqasympt3} yields 
\begin{align*}
    \xi_0^d(\Psi^j) \|\Psi^j\|_{L^2(\mathbb{S}^{d-1})}^2 &= (1+ \mathcal{O}(N_j^{-2})) \|\Psi^j \|_{L^2(\mathbb{S}^{d-1})}^2 + \mathcal{O}(N_j^{(d-6)/2}) \sum_{n=M_j}^{N_j}\sum_{k \in \mathcal{I}_n^d}  \Psi^j(n, k).
\end{align*}
In this case, the fact that $(\Psi^j)_{j=0}^\infty$ is a frame-generating sequence implies the inequality $\lvert \Psi^j(n, k)\rvert \leq c N_j^{(d-2)/2}$ for $M_j \leq n \leq N_j$. Thus, utilizing \ref{C1}, we have
\begin{equation*}
    \sum_{n=M_j}^{N_j}\sum_{k \in \mathcal{I}_n^d} \lvert \Psi^j(n, k) \rvert \leq c N_j^{d/2} \leq c N_j^{-(d-2)/2} \| \Psi^j \|_{L^2(\mathbb{S}^{d-1})}^2
\end{equation*}
and, therefore,
\begin{equation}\label{eqzetaapprox}
    \xi_0^d(\Psi^j) = 1 + \mathcal{O}(N_j^{-2}).
\end{equation}

At this point, we have shown that \eqref{eqzetaapprox} holds under the assumptions given in \hyperref[proplocalization]{Proposition~\ref*{proplocalization}}. Hence, we obtain $1-cN_j^{-2} \leq \|\xi_0(\Psi^j) \|_2^2 \leq 1$ and, therefore,
\begin{equation*}
    \mathrm{Var}_{\mathrm{S}}(\Psi^j) \leq c N_j^{-2}.
\end{equation*}
This completes the proof.
\end{proof}
\begin{remark}
    We note that the conclusion of \hyperref[proplocalization]{Proposition~\ref*{proplocalization}} remains valid under a different set of assumptions. Indeed, assume that $(\Psi^j)_{j=0}^\infty$ is merely a steerable sequence with $\Psi^j \in \Pi_{N_j}(\mathbb{S}^{d-1})$ for each $j \in \mathbb{N}_0$, and that condition \ref{C2} holds.  Assume further that, for all 
$M_j \leq n \leq N_j$, 
    \begin{equation}\label{C4}
        \left \vert \frac{\Psi^j(n+1, k) + \Psi^j(n-1, k)}{2} - \Psi^j(n, k) \right\vert \leq cN_j^{-2} \, \lvert\Psi^j(n, k)\rvert,
    \end{equation}
    where the constant $c>0$ does not depend on $n, k$ or $j$. Under this assumption, we obtain 
\begin{equation*}
   \sum_{n=M_j}^{N_j} \sum_{k \in \mathcal{I}_n^d} \Psi^j(n, k) \frac{\overline{\Psi^j(n+1, k)} + \overline{\Psi^j(n-1, k)}}{2} = \|\Psi^j\|_{L^2(\mathbb{S}^{d-1})}^2(1 +  \mathcal{O}(N_j^{-2})),
\end{equation*}
and consequently, \eqref{eqzetaapprox} holds. Therefore, the bounds \eqref{localization eq1} and \eqref{localization eq2} hold.  
\end{remark}

In \autoref{sec5}, we use \hyperref[proplocalization]{Proposition~\ref*{proplocalization}} to present explicit constructions of directional polynomial frames $\mathcal{X}[(\Psi^j)_{j=0}^\infty]$ that are optimally localized in the sense of \eqref{eqoptimal}.

\subsection{Directional Sensitivity}\label{subsec4.4}
Let $(\Psi^j)_{j=0}^\infty \subset \Pi(\mathbb{S}^{d-1})$ be a sequence of polynomials that are well localized at the North Pole. In this case, the frame coefficients
\begin{equation}\label{inner products eq1}
    \langle f, T^d(g_\eta h)\Psi^j\rangle_{\mathbb{S}^{d-1}}, \quad h \in SO(d-1),
\end{equation}
display information about the signal $f \in L^2(\mathbb{S}^{d-1})$ at the position $\eta\in \mathbb{S}^{d-1}$. Depending on the sequence $(\Psi^j)_{j=0}^\infty$, this data can exhibit various complexities. For example, if $(\Psi^j)_{j=0}^\infty$ is $SO(d-1)$-invariant, then the inner products in \eqref{inner products eq1} are constant in $h$. More generally, \hyperref[propSteerable]{Proposition~\ref*{propSteerable}} implies that for a $K$-steerable sequence $(\Psi^j)_{j=0}^\infty$ the frame coefficients \eqref{inner products eq1} constitute a polynomial of degree $K$ on $SO(d-1)$.

In order to assess different systems in terms of the information \eqref{inner products eq1} they provide at each position $\eta\in \mathbb{S}^{d-1}$, we introduce a measure of directional sensitivity. As proposed in \cite{bib21} for the setting of directional wavelets on $\mathbb{S}^2$, we say that a function $\Psi^j \in \Pi(\mathbb{S}^{d-1})$, which is concentrated at $e^d$, exhibits a high directional sensitivity if the auto-correlation function
\begin{equation}\label{auto-correlation}
   SO(d-1) \rightarrow\mathbb{C}, \qquad h \mapsto  \langle T^d(h)\Psi^j, \Psi^j \rangle_{\mathbb{S}^{d-1}},
\end{equation}
is well localized at the identity matrix $\mathrm{I}_{d}$. 

In the case where $(\Psi^j)_{j=0}^\infty$ is $K$-steerable, the auto-correlation functions \eqref{auto-correlation} are polynomials of degree $K$.
Thus, increasing $K$ allows for stronger directionality, whereas $0$-steerable frames (i.e., zonal frames) exhibit no directional sensitivity. 

As demonstrated in \cite{bib9, bib10, bib21} and \cite{bib37}, optimizing the localization of the auto-correlation functions \eqref{auto-correlation} represents a powerful tool for designing directional polynomial frames. In \autoref{sec5}, we will apply this method to certain systems that are optimally localized in the sense of \autoref{subsec4.3}. Applications of localized directional frames include, in particular, the detection of edges (or boundary pixels) in signals. Specifically, directional frames on $\mathbb{S}^2$ have been utilized in the search for cosmic strings in the cosmic microwave background \cite{bib76, bib67, bib77} and in the context of wavelet-based segmentation \cite{bib68}. For a more theoretical discussion, we refer to \cite{bib26, bib31}, where it was proven that directional wavelets \cite{bib9, bib10, bib21} and second-generation curvelets \cite{bib1} can identify edges in signals not only by their position but also by their orientation.

\subsection{A Note on the Quadrature Formulas}\label{subsec4.5}
In the following, we will demonstrate how integration rules \eqref{eq35} can be constructed from spherical quadrature formulas. In contrast to the $SO(d)$ setting, such formulas are well investigated in the literature (see, e.g., \cite{bib34, bib15, bib72, bib73, bib74, bib71} and the references therein). After reviewing the general case, we also discuss relevant special cases where the sequence $(\Psi^j)_{j=0}^\infty$ exhibits certain structural properties that allow for particularly simple discretization grids. We note that all results given in this article are still valid when replacing the original systems $ \mathcal{X}[(\Psi^j)_{j=0}^\infty]$ w.r.t.\ the quadrature rules \eqref{eq35} by the corresponding modifications presented in this subsection. Here, merely some obvious notational adjustments are required.

\subsubsection{The General Case}
We consider a sequence of positive quadrature rules
\begin{equation}\label{quadrature S^(d-1)}
    \int_{\mathbb{S}^{d-1}} f \, \mathrm{d}\omega_{d-1} = \sum_{r=1}^{r_j} \omega_{j, r} f(\eta_{j, r}) \qquad \text{for each } f \in \Pi_{2N_j}(\mathbb{S}^{d-1}), \quad j \in \mathbb{N}_0,
\end{equation}
with points $\eta_{j, r}\in \mathbb{S}^{d-1}$ and corresponding weights $\omega_{j, r}>0$. Moreover, let $h_{j, s}\in SO(d-1)$ and $w_{j, s}>0$ be given such that
\begin{equation}\label{quadrature SO(d-1)}
    \int_{SO(d-1)} f_1 f_2 \, \mathrm{d}\mu_{d-1} = \sum_{s=1}^{s_j} w_{j, s} f_1(h_{j, s}) f_2(h_{j, s}) \quad \text{if } f_1, f_2 \in \mathcal{M}_{N_j}^1(SO(d-1)), \quad j \in \mathbb{N}_0.
\end{equation}
Using \eqref{eq2}, \hyperref[lemma4]{Lemma~\ref*{lemma4}} and \hyperref[lemma product property]{Lemma~\ref*{lemma product property}}, we obtain
\begin{equation*}
    \int_{SO(d)} f_1 f_2 \, \mathrm{d}\mu_{d} = \sum_{r=1}^{r_j} \sum_{s=1}^{s_j} \omega_{j, r} w_{j, s} f_1(g_{\eta_{j, r}} h_{j, s}) f_2(g_{\eta_{j, r}} h_{j, s}) \quad \text{if } f_1, f_2 \in \mathcal{M}_{N_j}^1(SO(d)).
\end{equation*}
Hence, the systems $\mathcal{X}[(\Psi^j)_{j=0}^\infty]$ defined in \eqref{ourFrames} can be replaced by
\begin{equation}\label{ourFrames general case}
    \mathcal{X}[(\Psi^j)_{j=0}^\infty] = \{ \sqrt{\omega_{j, r} w_{j, s}} \, T^d(g_{\eta_{j, r}} h_{j, s})\Psi^j, \; r=1, \dots, r_j, \; s=1, \dots, s_j, \; j \in \mathbb{N}_0\}.
\end{equation}


Repeating the above process, we can now replace the formulas \eqref{quadrature SO(d-1)} with suitable quadrature rules on $\mathbb{S}^{d-2}$ and $SO(d-2)$, and so on. In this way, integration formulas \eqref{eq35} can be obtained solely from positive quadrature rules for polynomials on $\mathbb{S}^{m}$, $m=1, \dots, d-1$. Moreover, in many relevant situations these formulas can be simplified significantly, as we will discuss in the following.

\subsubsection{Steerable Frames}
Let $(\Psi^j)_{j=0}^\infty\subset \Pi(\mathbb{S}^{d-1})$ be a $K$-steerable sequence such that $\Psi^j \in \Pi_{N_j}(\mathbb{S}^{d-1})$ for each $j \in \mathbb{N}_0$. In this setting, instead of the sequence of quadrature rules \eqref{quadrature SO(d-1)}, it suffices to consider a single formula
\begin{equation}\label{quadrature SO(d-1) K-steerable}
\sum_{p=1}^M w_p \, f_1(h_p) f_2(h_p) = \int_{SO(d-1)} f_1 f_2 \, \mathrm{d}\mu_{d-1} \qquad \text{if } f_1, f_2 \in \mathcal{M}_K^1(SO(d-1)),
\end{equation}
where the $h_p\in SO(d-1)$ and $w_p>0$ are suitable points and weights. Indeed, \hyperref[propSteerable]{Proposition~\ref*{propSteerable}} shows that under this assumption, the system defined in \eqref{ourFrames} can be replaced  by 
\begin{equation*}\label{K-steerable frame}
    \mathcal{X}[(\Psi^j)_{j=0}^\infty] = \{\sqrt{\omega_{j, r} w_{ p}}\, T^d(g_{\eta_{j, r}} h_p)\Psi^j, \; p=1,\dots, M, \; r=1, \dots, r_j, \; j \in \mathbb{N}_0 \},
\end{equation*}
where the quadrature points $\eta_{j, r}\in \mathbb{S}^{d-1}$ and weights $\omega_{j, r}>0$ are chosen as in \eqref{quadrature S^(d-1)}.

As a consequence, the number of rotations required at each scale $j$ is significantly reduced compared to the general case. In fact, the computational cost involved is comparable to that of the zonal setting. Also, as discussed above, a formula of the form \eqref{quadrature SO(d-1) K-steerable} can be constructed from quadrature rules for $\Pi_K(\mathbb{S}^{d-2})$, \dots, $\Pi_K(\mathbb{S}^1)$ by repeated application of \hyperref[lemma4]{Lemma~\ref*{lemma4}} and \hyperref[lemma product property]{Lemma~\ref*{lemma product property}}.

\subsubsection{Rotation-Invariant Frames}
$SO(m)$-invariant sequences $(\Psi^j)_{j=0}^\infty$ provide another setting in which the quadrature grids used in \eqref{ourFrames general case} can be simplified. In the following, the two most symmetric cases, namely $m=d-1$ and $m=d-2$, are discussed in detail.

If $(\Psi^j)_{j=0}^\infty$ is $SO(d-1)$-invariant with $\Psi^j \in \Pi_{N_j}(\mathbb{S}^{d-1})$ for each $j \in \mathbb{N}_0$, it is easy to see that \eqref{ourFrames general case} can be replaced by
\begin{equation*}
   \mathcal{X}[(\Psi^j)_{j=0}^\infty] =  \{\sqrt{\omega_{j, r}}\,  T^d(g_{\eta_{j, r}})\Psi^j, \; r=1, \dots, r_j, \; j \in \mathbb{N}_0 \},
\end{equation*}
where the points $\eta_{j, r}\in \mathbb{S}^{d-1}$ and weights $\omega_{j, r}>0$ are chosen as in \eqref{quadrature S^(d-1)}. Frames of this form coincide with the well-established zonal/isotropic systems.

Now, let $(\Psi^j)_{j=0}^\infty$ be a $SO(d-2)$-invariant sequence with $\Psi^j \in \Pi_{N_j}(\mathbb{S}^{d-1})$ for each $j \in \mathbb{N}_0$. Let $\eta_{j, r}\in \mathbb{S}^{d-1}$ and $\omega_{j, r}>0$ be as in \eqref{quadrature S^(d-1)}. In addition, let $\eta_{j, s}' \in \mathbb{S}^{d-2}$ and  $\omega_{j, s}'>0$ be quadrature points and weights satisfying
\begin{equation}\label{quadrature S^(d-2)}
    \int_{\mathbb{S}^{d-2}} f \, \mathrm{d}\omega_{d-2} = \sum_{s=1}^{s_j} \omega_{j, s}' f(\eta_{j, s}') \qquad \text{for each } f \in \Pi_{2N_j}(\mathbb{S}^{d-2}).
\end{equation}
For $\eta' \in \mathbb{S}^{d-2}$, let $h_{\eta'}$ denote any rotation matrix in $SO(d-1)\subset SO(d)$ satisfying $$h_{\eta'} e^{d-1} = (\eta', 0)\in \mathbb{S}^{d-1}.$$ In this setting, it is not difficult to verify that  the systems \eqref{ourFrames general case} can be replaced by
\begin{equation*}
    \mathcal{X}[(\Psi^j)_{j=0}^\infty] = \left\{\sqrt{\omega_{j, r} \omega_{j, s}'}\,  T^d(g_{\eta_{j, r}} h_{\eta_{j, s}'})\Psi^j, \; r=1, \dots, r_j,\; s=1, \dots, s_j, \; j \in \mathbb{N}_0 \right\}.
\end{equation*}

\subsubsection{Steerable and Rotation-Invariant Frames}\label{subsubsec4.5.4}
If $(\Psi^j)_{j=0}^\infty$ is both $K$-steerable and $SO(d-2)$-invariant, a combination of the ideas presented above leads to quadrature measures of a particularly simple structure. Namely, let $\eta_{j, r} \in \mathbb{S}^{d-1}$, $\omega_{j, r}>0$ be as in \eqref{quadrature S^(d-1)} and, additionally, $\eta_{p}'\in \mathbb{S}^{d-2}$ with corresponding weights $w_p>0$ such that
\begin{equation}\label{eq2 quadrature S^d-2}
     \int_{\mathbb{S}^{d-2}} f \, \mathrm{d}\omega_{d-2} = \sum_{p=1}^{M} w_{ p} f(\eta_{ p}') \qquad \text{for each } f \in \Pi_{2K}(\mathbb{S}^{d-2}).
\end{equation}
Then, \eqref{ourFrames general case} can be replaced with
\begin{equation*}\label{steerable SO(d-2)-inv frame}
   \mathcal{X}[(\Psi^j)_{j=0}^\infty] = \{\sqrt{\omega_{j, r} w_p}\,  T^d(g_{\eta_{j, r}} h_{\eta_{p}'})\Psi^j, \; r=1, \dots, r_j,\; p=1, \dots, M, \; j \in \mathbb{N}_0 \}.
\end{equation*}

\section{Examples}\label{sec5}
The results derived in this article pave the way for a variety of new polynomial frames for $L^2(\mathbb{S}^{d-1})$, $d \geq 3$, which are both directional and localized. In this section, we present specific examples that we consider particularly interesting. Our constructions here are based exclusively on the orthonormal basis of spherical harmonics and the corresponding index set described in \hyperref[appendix A]{Appendix A}.

\subsection{Directional Wavelets}
In order to obtain frames $\mathcal{X}[(\Psi^j)_{j=0}^\infty]$ for $L^2(\mathbb{S}^{d-1})$ that are optimally localized in terms of the uncertainty principle \eqref{goh_goodman2}, we will first need to construct a suitable sequence of frequency filters. Let $\phi \in C^2([0, \infty))$ be a real-valued and non-increasing function with $\supp(\phi) = [0, 1]$ and $\phi(t)=1$ for each $t \in [0, 1/2]$. Then the map
\begin{equation*}
    \kappa\colon [0, \infty) \rightarrow \mathbb{R}, \qquad \kappa(t) = \sqrt{\phi^2(t/2)-\phi^2(t)}
\end{equation*}
is twice continuously differentiable with $\supp(\kappa) = [1/2, 2]$. To give an explicit example, we choose
\begin{equation*}
     \phi(t) = \begin{cases}
        1, \quad &t\in [0, 1/2),\\
        16(1-t)^3(12t^2-9t+2),  &t \in [1/2, 1],\\
        0 & t>1.
    \end{cases}
\end{equation*}
In this case, both $\phi$ and $\kappa$ are twice, but not three times, continuously differentiable. We can now define a sequence of window functions by
\begin{equation}\label{window1}
    \kappa_{1,j}\colon [0, \infty) \rightarrow \mathbb{R}, \quad   \kappa_{1,j}(n) = 2^{j(d-2)/2} \kappa\! \left( \frac{n}{2^{j-1}}\right), \qquad j \geq 1.
\end{equation}
Alternatively, we will also consider the windows
\begin{equation*}\label{window2}
    \kappa_{2,j}\colon [0, \infty) \rightarrow \mathbb{R}, \quad \kappa_{2, j}(n) =\begin{cases} 2^{j(d-2)/2} \sin\!\left(\pi \frac{n+1-2^{j-2}}{3\cdot2^{j-2}+2} \right), \quad &2^{j-2}\leq n \leq 2^j,\\
    0, & \text{else},
    \end{cases} \qquad j \geq 1.
\end{equation*}
We note that in a previous article \cite{bib11} functions of the form $\kappa_{2, j}$ were utilized for the construction of zonal frames that are optimally localized in terms of a spherical uncertainty principle \cite{bib45, bib42, bib38}.

In addition to the window functions defined above, we will also utilize certain directionality components $\zeta_{k}^{d, n} \in \mathbb{C}$. For a fixed $K\in \mathbb{N}_0$, which corresponds to the desired degree of directional sensitivity, such components can be used to ensure $K$-steerability of the corresponding system and to optimize the directionality under this constraint. As proposed in \cite{bib37}, for $d \geq 4$ we define
\begin{equation*}\label{zeta formula}
 \zeta_k^{d,n}=\delta_{k_2, 0} \begin{cases}
 0, \quad \text{if }k_1>K_n \text{ or if } K_n-k_1 \text{ is odd},  \\
\displaystyle(-1)^{\lfloor \frac{k_1}{2} \rfloor}  \sqrt{ \frac{\Gamma(\lambda) K_n !(k_1+\lambda) \Gamma(d+k_1-3)}{\Gamma(2\lambda)2^{K_n}(\frac{K_n-k_1}{2})! \Gamma(\lambda +\frac{K_n+k_1}{2}+1) k_1!}},  \quad\text{else},
\end{cases}
\end{equation*}

where $K_n = \min(K, n)$ and $\lambda = (d-3)/2$. This choice yields, in particular,
\begin{equation*}
    \sum_{k \in \mathcal{I}_n^d} \lvert  \zeta_k^{d,n} \rvert^2 = 1 \quad \text{for each } n \in \mathbb{N}.
\end{equation*}
For directionality components in the case $d=3$, we refer to \cite{bib21, bib9}. 

Let $d \geq 4$. We consider the usual setting $N_0 =0$ and $N_j = 2^{j}$ for $j \in \mathbb{N}$. With respect to the orthonormal basis of spherical harmonics given in \hyperref[appendix A]{Appendix A}, we define the polynomial sequences $(\Psi_1^j)_{j=0}^\infty$ and  $(\Psi_2^j)_{j=0}^\infty$ via
\begin{equation*}
    \Psi_i^0 \equiv 1, \qquad \Psi_i^j(n, k) = \kappa_{i, j}(n) \zeta_{k}^{d, n}, \quad j \in \mathbb{N}, \qquad i \in \{1, 2 \}.
\end{equation*}
Then, clearly, $\Psi_i^j \in \Pi_{N_j}(\mathbb{S}^{d-1})$ for each $j \in \mathbb{N}_0$. By \hyperref[lemmasteerable]{Lemma~\ref*{lemmasteerable}} and by \eqref{k_(d-m)>0}, these sequences are $K$-steerable and $SO(d-2)$-invariant. Moreover, it is easy to verify that $(\Psi_1^j)_{j=0}^\infty$ and  $(\Psi_2^j)_{j=0}^\infty$ satisfy the admissibility condition \eqref{eq32}. Hence, according to the discussion in \autoref{subsubsec4.5.4}, both systems
\begin{equation*}
   \mathcal{X}[(\Psi_i^j)_{j=0}^\infty] =\{\sqrt{\omega_{j, r} w_p}\,  T^d(g_{\eta_{j, r}} h_{\eta_{p}'})\Psi_i^j, \; r=1, \dots, r_j,\; p=1, \dots, M, \; j \in \mathbb{N}_0 \},
\end{equation*}
for $i \in \{1, 2 \}$, constitute a frame for $L^2(\mathbb{S}^{d-1})$. Here, the points $\eta_{j, r}\in \mathbb{S}^{d-1}$, $\eta_p'\in \mathbb{S}^{d-2}$ and nonnegative weights $\omega_{j, r}$, $w_p$ are as in \eqref{quadrature S^(d-1)} and \eqref{eq2 quadrature S^d-2}, respectively.

With regard to the assumptions of \hyperref[proplocalization]{Proposition~\ref*{proplocalization}}, it is easy to see that $(\Psi_1^j)_{j=0}^\infty$ and  $(\Psi_2^j)_{j=0}^\infty$ satisfy \ref{C1} and \ref{C2} for the standard setting $N_0 = M_0 =0$ and $N_j = 2^{j}$, $M_j= \max(1, 2^{j-2})$ for $j \in \mathbb{N}$. Moreover, a straightforward calculation shows that $(\Psi_1^j)_{j=0}^\infty$ satisfies condition \ref{C3}, whereas $(\Psi_2^j)_{j=0}^\infty$ satisfies \eqref{C4}. Hence, both frames $\mathcal{X}[(\Psi_1^j)_{j=0}^\infty]$ and $\mathcal{X}[(\Psi_2^j)_{j=0}^\infty]$ are optimally localized in terms of \eqref{goh_goodman2}, i.e.,
\begin{equation*}
    \mathrm{Var}_{\mathrm{S}}(\Psi_i^j) \sim 2^{-2j}, \qquad i \in \{1, 2 \}.
\end{equation*}
Furthermore, both frames exhibit an optimal directional sensitivity (of degree $K$), in the sense discussed in \autoref{subsec4.4}, as
\begin{equation}\label{highly directional}
    \langle T^d(h) \Psi_i^j,  \Psi_i^j\rangle_{\mathbb{S}^{d-1}}= \sum_{n=0}^{\infty}\lvert \kappa_{i, j}(n)\rvert^2 (\langle e^{d-1} , h e^{d-1}\rangle)^{\min(K, n)}, \quad h \in SO(d-1).
\end{equation}
For more details on the calculations that lead to \eqref{highly directional} we refer to \cite{bib37}.

We note that the frames $\mathcal{X}[(\Psi_1^j)_{j=0}^\infty]$ and $\mathcal{X}[(\Psi_2^j)_{j=0}^\infty]$ are constructed similarly to the directional wavelets in \cite{bib9} and \cite{bib37}. However, the pointwise localization estimates established for those wavelets require a high order of smoothness in the Fourier domain, which excludes the window function defined above.

As pointed out in \cite{bib37}, there is a convenient way to visualize functions that are $SO(d-2)$-invariant and concentrated at the North Pole, such as $\Psi_i^j$. To this end, we associate to each element $\xi$ in the unit tangent space $UT_{e^d}\mathbb{S}^{d-1} = \{(\eta', 0) \mid \eta' \in \mathbb{S}^{d-2}\}$ the corresponding curve
\begin{equation*}
    t \mapsto \cos t \, e^d + \sin t \, \xi, \quad t \in [0, \pi],
\end{equation*}
starting from the North Pole $e^d$. Clearly, $UT_{e^d}\mathbb{S}^{d-1}$ admits the parametrization 
\begin{equation*}
    \xi(\varphi, \eta'') = \cos \varphi \, e^{d-1} + \sin \varphi \, (0, 0, \eta''), \quad \varphi \in [0, 2\pi), \; \eta'' \in \mathbb{S}^{d-3}.
\end{equation*}
By the $SO(d-2)$-invariance of $\Psi_i^j$, it follows that the map
\begin{equation}\label{polarcoord beispiel1}
   \psi_i^j (t, \varphi) = \Psi_i^j(\cos t \, e^d + \sin t (\cos \varphi \, e^{d-1} + \sin \varphi \, (0, 0, \eta''))), \quad (t, \varphi) \in [0, \pi]\times[0, 2\pi),
\end{equation}
does not depend on $\eta'' \in \mathbb{S}^{d-3}$. Thus, as presented in \autoref{fig1} and \autoref{fig2} for $d=4$, the wavelets $\Psi_i^j$ can be visualized as functions on the unit ball $\mathbb{B}_\pi^2 = \{ x \in \mathbb{R}^2 \mid \| x \|_2 \leq \pi \}$ expressed in polar coordinates via \eqref{polarcoord beispiel1}. For visual clarity, all images have been re-scaled to take values between $-1$ and $1$.

\begin{figure}
\includegraphics[width=\textwidth]{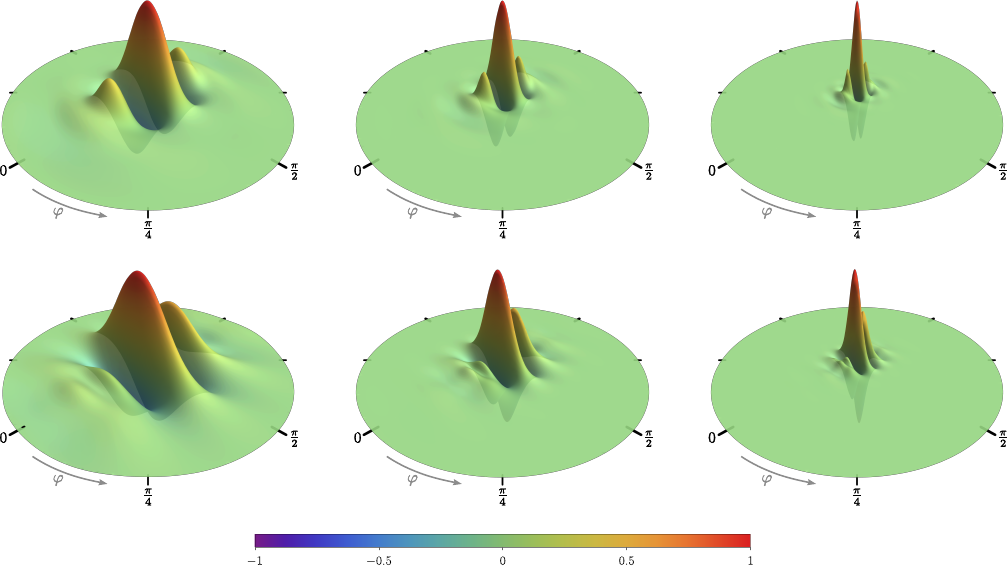}
\caption{Re-scaled directional wavelets $\psi_1^j (t, \varphi)$, $(t, \varphi) \in [0, 1]\times[0, 2\pi)$, for $K=4,9$ from top to bottom and $j=5, 6, 7$ from left to right}\label{fig1}
\end{figure}

\begin{figure}
\includegraphics[width=\textwidth]{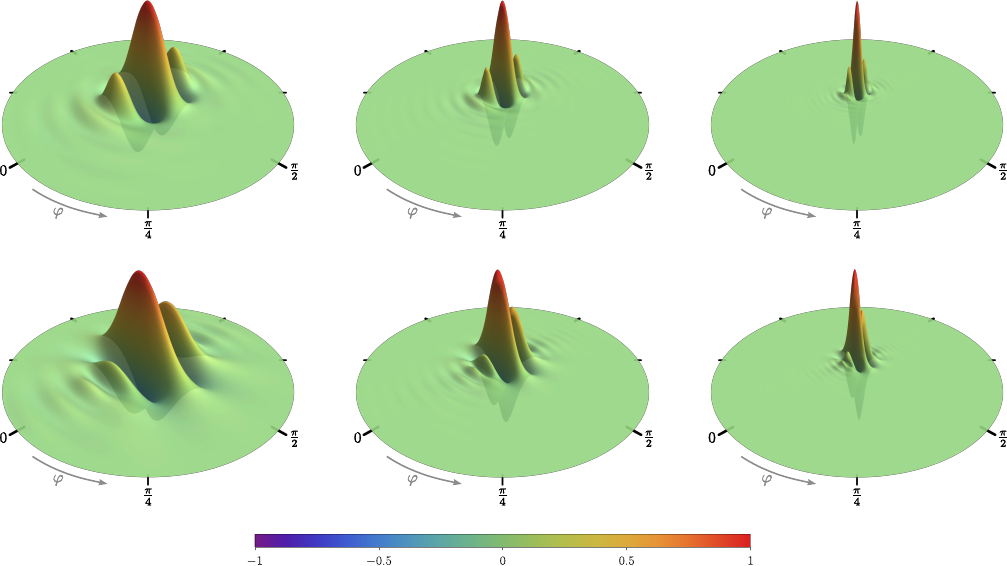}
\caption{Re-scaled directional wavelets $\psi_2^j (t, \varphi)$, $(t, \varphi) \in [0, 1]\times[0, 2\pi)$, for $K=4,9$ from top to bottom and $j=5, 6, 7$ from left to right}\label{fig2}
\end{figure}

\subsection{Curvelets}
As a second example, we construct frames that extend the two-dimensional second-genera\-tion curvelets introduced in \cite{bib1} to higher-dimensional spheres. Just like in the two-dimensional setting, these polynomial curvelets are not restricted in their directional sensitivity, giving them a clear advantage over competing systems such as spherical needlets and directional polynomial wavelets, which are steerable by design. Using the orthonormal basis of spherical harmonics given in \hyperref[appendix A]{Appendix A}, we define an initial sequence $(\Psi^j)_{j=0}^\infty$ as
\begin{equation*}
    \Psi^0 \equiv 1, \qquad  \Psi^j(n, k) = \kappa_{1,j}(n) \, \frac{\delta_{\lvert k_{d-2}\rvert, n}}{\sqrt{2}}, \quad j \in \mathbb{N}.
\end{equation*}
Here, the window functions $\kappa_{1,j}(n)$ are given by \eqref{window1}. A straightforward computation shows that
\begin{equation}\label{curvelet1}
    \Psi^j(\eta) = \sqrt{2}\sum_{n=0}^\infty \kappa_{1, j}(n) \, A_{(n, \dots, n)}^n \Re \{(\eta_2+ \mathrm{i}\eta_1)^n\}, \quad j \geq 1,
\end{equation}
where $\Re\{z\}$ denotes the real part of $z \in \mathbb{C}$. Clearly, \eqref{curvelet1} implies that $\Psi^j$ is concentrated at $e^2$. However, it might be convenient to work with a rotated version of the sequence $(\Psi^j)_{j=0}^\infty$ which is located at the North Pole. Thus, let $g_0 \in SO(d)$ be a rotation matrix such that $g_0 e^1 = e^{d-1}$, $g_0 e^2 = e^d$ and let
\begin{equation*}
    \Psi_{\scriptscriptstyle \mathrm{C}}^j = T^d(g_0)\Psi^j, \qquad j \in \mathbb{N}_0.
\end{equation*}
Clearly, we have
\begin{equation*}
    \Psi_{\scriptscriptstyle \mathrm{C}}^j(\eta) = \sqrt{2}\sum_{n=0}^\infty \kappa_{1, j}(n) \, A_{(n, \dots, n)}^n \Re \{(\eta_d+ \mathrm{i}\eta_{d-1})^n\}, \quad j \geq 1.
\end{equation*}
Consequently, the functions $\Psi_{\scriptscriptstyle \mathrm{C}}^j$ are localized at $e^d$ and the sequence $(\Psi_{\scriptscriptstyle \mathrm{C}}^j)_{j=0}^\infty$ is $SO(d-2)$-invariant. Moreover, it is easy to verify that $(\Psi_{\scriptscriptstyle \mathrm{C}}^j)_{j=0}^\infty$ satisfies the frame condition \eqref{eq32}. Thus, the system
\begin{equation*}
   \mathcal{X}[(\Psi_{\scriptscriptstyle \mathrm{C}}^j)_{j=0}^\infty] = \left\{\sqrt{\omega_{j, r} \omega_{j, s}'}\,  T^d(g_{\eta_{j, r}} h_{\eta_{j, s}'})\Psi_{\scriptscriptstyle \mathrm{C}}^j, \; r=1, \dots, r_j,\; s=1, \dots, s_j, \; j \in \mathbb{N}_0 \right\}
\end{equation*}
constitutes a frame for $L^2(\mathbb{S}^{d-1})$, where the points $\eta_{j, r}\in \mathbb{S}^{d-1}$, $\eta_{j, s}' \in \mathbb{S}^{d-2}$ and nonnegative weights $\omega_{j, r}$, $\omega'_{j, s}$ are as in \eqref{quadrature S^(d-1)} and \eqref{quadrature S^(d-2)}, respectively.

In \autoref{fig3}, the curvelets $\Psi_{\scriptscriptstyle \mathrm{C}}^j$ are visualized for $d=4$ in terms of the functions
\begin{equation*}
    \psi_{\scriptscriptstyle \mathrm{C}}^j (t, \varphi) = \Psi_{\scriptscriptstyle \mathrm{C}}^j(\cos t \, e^d + \sin t (\cos \varphi \, e^{d-1} + \sin \varphi \,e^{d-2})), \quad (t, \varphi) \in [0, \pi]\times[0, 2\pi).
\end{equation*}
Again, the presented images have been re-scaled to take values between $-1$ and $1$.

Under more restrictive assumptions, one can show that polynomial curvelets satisfy a strong direction-dependent localization bound, as previously observed in \cite[Proposition~2.4]{bib26} for the two-dimensional setting. The details will be presented in a future article.

\begin{figure}
\includegraphics[width=\textwidth]{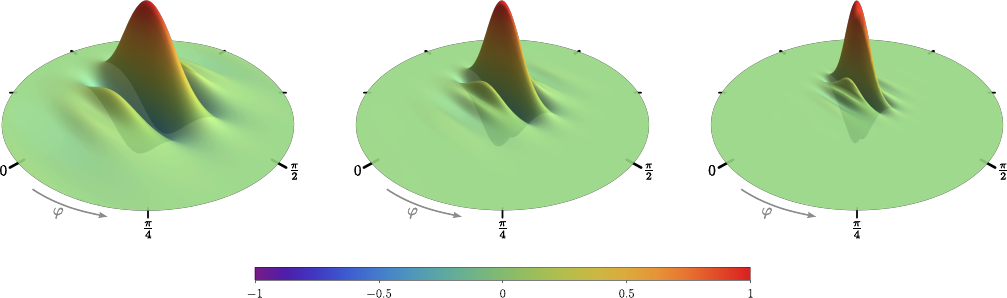}
\caption{Re-scaled curvelets $\psi_{\scriptscriptstyle \mathrm{C}}^j (t, \varphi)$, $(t, \varphi) \in [0, 1]\times[0, 2\pi)$, for $j=5, 6, 7$ from left to right}\label{fig3}
\end{figure}

\appendix
\addcontentsline{toc}{section}{Appendix}
\section*{Appendix}

\setcounter{equation}{0}\renewcommand\theequation{A.\arabic{equation}}



\subsection*{Appendix A: A Specific Orthonormal Basis of Spherical Harmonics}\label{appendix A}
The sphere $\mathbb{S}^{d-1}$ can be parameterized in terms of the spherical coordinates $\theta_1 \in [0, 2\pi)$, $\theta_2, \dots, \theta_{d-1}\in [0, \pi]$, via
\begin{equation*}
\eta(\theta_1, \theta_2, \dots, \theta_{d-1}) = \begin{pmatrix}
\sin \theta_{d-1}\, \cdots \, \sin \theta_2 \, \sin \theta_1 \\
\sin \theta_{d-1}\, \cdots \, \sin \theta_2 \, \cos \theta_1 \\
\sin \theta_{d-1}\,  \cdots\,  \sin \theta_3 \, \cos \theta_2 \\
\vdots \\
\sin \theta_{d-1} \, \cos \theta_{d-2} \\
\cos \theta_{d-1}
\end{pmatrix}.
\end{equation*}
In this setting, the (normalized) surface measure takes the form
\begin{equation*}
\mathrm{d}\omega_{d-1} = \frac{\Gamma(\frac{d}{2})}{2 \pi^{d/2}} \sin^{d-2}\theta_{d-1} \, \dots\,\sin \theta_2 \, \mathrm{d}\theta_1 \, \dots\, \mathrm{d}\theta_{d-1}.
\end{equation*}
With respect to spherical coordinates and for $d\geq 3$, a commonly used orthonormal basis (see, e.g., \cite[p.~466]{bib32}) is given by 
\begin{align}\label{y_n^k def}
Y_k^{d,n}(\theta_1, \dots, \theta_{d-1}) = A_k^n \prod_{j=0}^{d-3} C_{ k_j  - \lvert k_{j+1}\rvert}^{\frac{d-j-2}{2}+\lvert k_{j+1}\rvert}(\cos \theta_{d-j-1}) \, \sin^{\lvert k_{j+1}\rvert}(\theta_{d-j-1}) \, \exp(\mathrm{i}k_{d-2}\theta_1),
\end{align}
where  $k_0=n \in \mathbb{N}_0$ and $k=(k_1, \dots, k_{d-2})\in \mathcal{I}_n^d$ with
\begin{equation}\label{indexset}
\mathcal{I}_n^d = \{ (k_1, \dots, k_{d-2}) \in \mathbb{N}_0^{d-3}\times \mathbb{Z} : n\geq k_1 \geq \dots\geq k_{d-3}\geq \lvert k_{d-2}\rvert \}.
\end{equation}
The one-dimensional Gegenbauer polynomials $C_m^\lambda$ occurring in \eqref{y_n^k def} are defined as in \eqref{gegenbauer polynomials} and the normalization factor $A_k^n >0$ satisfies
\begin{align}\label{eq10}
 (A_k^n)^2 = \frac{2^{(d-4)(d-2)}}{\Gamma\!\left(\frac{d}{2}\right)} \prod_{j=0}^{d-3}\frac{2^{2\lvert k_{j+1}\rvert-j}(k_j-\lvert k_{j+1}\rvert)!(2k_j+d-j-2)\Gamma^2(\frac{d-j-2}{2}+\lvert k_{j+1}\rvert)}{\sqrt{\pi}\Gamma(k_j+\lvert k_{j+1}\rvert+d-j-2)}.
\end{align}
  For each $n \in \mathbb{N}_0$, $Y_k^{d, n}$ is a spherical harmonic of degree $n$ and the system $\{ Y_k^{d, n}, \; k \in \mathcal{I}_n^d\}$ is an orthonormal basis for $\mathcal{H}_n^d$. Consequently, $\{Y_k^{d, n}, \; k \in \mathcal{I}_n^d, \; n \in \mathbb{N}_0 \}$ is an orthonormal basis for $L^2(\mathbb{S}^{d-1})$.

\setcounter{equation}{0}\renewcommand\theequation{B.\arabic{equation}}

\subsection*{Appendix B: Proof of {\hyperref[lemmaUP]{Lemma~\ref*{lemmaUP}}}}\label{appendix B}
\begin{proof}
    In order to prove the assertion, we follow the methods in \cite{bib11, bib41} and start by considering the $d$-th component of the vector $\xi_0(f)$. 
    First, we obtain
    \begin{align*}
    \xi_0^d(f) = \int_{\mathbb{S}^{d-1}} x_d|f(x)|^2\, \mathrm{d}\omega_{d-1}(x) = \sum_{n = 0}^{\infty}\sum_{n' = 0}^{\infty}\sum_{k \in \mathcal{I}_n^d}\sum_{k' \in \mathcal{I}_{n'}^d} f(n,k)\,\overline{f(n',k')}\,I(n,n',k,k')
    \end{align*}
    by setting
   \begin{align}\label{I_n}
    I(n,n',k,k') = \int_{\mathbb{S}^{d-1}} x_dY_k^{d,n}(x)\overline{Y_{k'}^{d,n'}(x)}\, \mathrm{d}\omega_{d-1}(x)
    \end{align}
    and expanding $f$ into its Fourier series.
    In terms of spherical coordinates, substituting \eqref{y_n^k def} into \eqref{I_n} then yields
    \begin{align*}
    	& I(n,n',k,k') \notag \\
        & \qquad = \frac{\Gamma(\frac{d}{2})}{2\pi^{d/2}} A_k^nA_{k'}^{n'} \int_0^{2\pi} \exp(\mathrm{i}\theta_1(k_{d-2}-k'_{d-2})) \, \mathrm{d}\theta_1 \\
    	& \qquad \times \int_0^{\pi} C_{k_{d-3}-|k_{d-2}|}^{\frac{1}{2}+|k_{d-2}|}(\cos{\theta_2})C_{k'_{d-3}-|k'_{d-2}|}^{\frac{1}{2}+|k'_{d-2}|}(\cos{\theta_2})\sin^{|k_{d-2}|+|k'_{d-2}|+1}(\theta_2) \, \mathrm{d}\theta_2 \times \dots \\
    	& \qquad \times \int_0^{\pi} \cos{\theta_{d-1}}C_{n-|k_1|}^{\frac{d-2}{2}+|k_1|}(\cos{\theta_{d-1}})C_{n'-|k'_1|}^{\frac{d-2}{2}+|k'_1|}(\cos{\theta_{d-1}})\sin^{|k_1|+|k'_1|+d-2}{(\theta_{d-1})}\, \mathrm{d}\theta_{d-1}.
    \end{align*}
    We observe that $I(n,n',k,k')$ contains $d-3$ integrals which, after the substitution $t = \cos{\theta_{d-1-j}}$, take the form
    \begin{align}\label{d-3_int}
        \int_{-1}^1 C_{k_{j}-|k_{j+1}|}^{\frac{d-j-2}{2}+|k_{j+1}|}(t)\,C_{k'_{j}-|k'_{j+1}|}^{\frac{d-j-2}{2}+|k'_{j+1}|}(t)(1-t^2)^{\,(|k_{j+1}|+|k'_{j+1}|+d-j-3)/2} \, \mathrm{d}t, \quad j = 1, \dots, d-3.
    \end{align}
    Taking the factor $\delta_{k_{d-2}, k'_{d-2}}$ from the first integral into account and making repeated use of the orthogonality relation for Gegenbauer polynomials
    \begin{align}\label{ortho}
        \int_{-1}^1 C_n^{\lambda}(t)C_m^{\lambda}(t)(1-t^2)^{\, \lambda-\frac{1}{2}}\, \mathrm{d}t = \delta_{n,m} \int_{-1}^1|C_n^{\lambda}(t)|^2 (1-t^2)^{\, \lambda-\frac{1}{2}}\, \mathrm{d}t
    \end{align}
    as given in \cite[p.~418]{bib3}, we find that \eqref{d-3_int} reduces to
    \begin{align*}
        \delta_{k_{j},k'_{j}} \int_{-1}^1 |C_{k_{j}-|k_{j+1}|}^{\frac{d-j-2}{2}+|k_{j+1}|}(t)|^2(1-t^2)^{\,|k_{j+1}|+(d-j-3)/2}\, \mathrm{d}t, \quad j = 1,\dots, d-3.
    \end{align*}
    For the last integral, we now apply the three-term recurrence relation for Gegenbauer polynomials 
    \begin{align*}
        t\,C_{n-|k_1|}^{\frac{d-2}{2}+|k_1|}(t) = \frac{n-|k_1|+1}{2n+d-2}\,C_{n-|k_1|+1}^{\frac{d-2}{2}+|k_1|}(t) - \frac{n+|k_1|+d-3}{2n+d-2}C_{n-|k_1|-1}^{\frac{d-2}{2}+|k_1|}(t)
    \end{align*}
    as presented in \cite[p.~418]{bib3}.
    Together with \eqref{ortho}, this leads to
    \begin{align*}
        & \int_{-1}^{1} t\,C_{n-|k_1|}^{\frac{d-2}{2}+|k_1|}(t)\,C_{n'-|k_1|}^{\frac{d-2}{2}+|k'_1|}(t)(1-t^2)^{\,|k_1|+(d-3)/2}\, \mathrm{d}t  \\
    	& \qquad \qquad = \delta_{n+1,n'}\frac{n-|k_1|+1}{2n+d-2} \int_{-1}^1 |C_{n-|k_1|+1}^{\frac{d-2}{2}+|k_1|}(t)|^2 (1-t^2)^{\,|k_1|+(d-3)/2} \, \mathrm{d}t \\
    	& \qquad \qquad + \delta_{n-1,n'} \frac{n+|k_1|+d-3}{2n+d-2} \int_{-1}^1 |C_{n-|k_1|-1}^{\frac{d-2}{2}+|k_1|}(t)|^2 (1-t^2)^{\,|k_1|+(d-3)/2} \, \mathrm{d}t.
    \end{align*}
    Hence, we derive
    \begin{align*}
		& I(n,n',k,k')  = \frac{\Gamma(\frac{d}{2})}{\pi^{(d-2)/2}} \delta_{k,k'} \\
		& \times \left(\delta_{n+1,n'}A_k^nA_k^{n+1} \frac{n+|k_1|+d-3}{2n+d-2} \prod_{j=0}^{d-3} \int_{-1}^1 |C_{k_j-|k_{j+1}|}^{\frac{d-j-2}{2}+|k_{j+1}|}(t)|^2(1-t^2)^{\, |k_{j+1}|+ (d-j-3)/2} \, \mathrm{d}t \right. \\
		& \quad \;\left. + \delta_{n-1,n'}A_k^nA_k^{n-1}\frac{n+|k_1|+d-3}{2n+d-2} \prod_{j=0}^{d-3} \int_{-1}^1 |C_{k_j-|k_{j+1}|}^{\frac{d-j-2}{2}+|k_{j+1}|}(t)|^2(1-t^2)^{\, |k_{j+1}|+ (d-j-3)/2} \, \mathrm{d}t\right). 
	\end{align*}
    According to \cite[p.~437]{bib32}, the normalization factor $A_k^n > 0$ satisfies
    \begin{align*}
        \frac{1}{(A_k^n)^2} = \frac{\Gamma(\frac{d}{2})}{\pi^{(d-2)/2}} \prod_{j = 0}^{d-3} \int_{-1}^{1} |C_{k_j-|k_{j+1}|}^{\frac{d-j-2}{2}+|k_{j+1}|}(t)|^2(1-t^2)^{\, |k_{j+1}|+ (d-j-3)/2} \, \mathrm{d}t,
    \end{align*}
    which, when combined with \eqref{eq10}, gives
    \begin{align*}
        &\left(\frac{A_k^n}{A_k^{n+1}}\right)^2 = \frac{(2n+d-2)(n+|k_1|+d-2)}{(2n+d)(n+1-|k_1|)}, \\ 
		\intertext{and}
		&\left(\frac{A_k^n}{A_k^{n-1}}\right)^{2} = \frac{(2n+d-2)(n-|k_1|)}{(2n+d-4)(n+|k_1|+d-3)}.
    \end{align*}
    It follows that
    \begin{align}\label{I_2}
    I(n,n',k,k') &= \delta_{k,k'} \left(\delta_{n+1,n'} \sqrt{\frac{(n+|k_1|+d-2)(n-|k_1|+1)}{(2n+d)(2n+d-2)}} \right. \notag \\
    &\qquad \qquad + \left. \delta_{n-1,n'}\sqrt{\frac{(n+|k_1|+d-3)(n-|k_1|)}{(2n+d-2)(2n+d-4)}}\right).
    \end{align}
    By inserting \eqref{I_2} into $\xi_0(f)$, it is now easy to verify that \eqref{eqUP} holds.
\end{proof}



\sloppy 

\section*{Literature}       
\addcontentsline{toc}{section}{Literature}
\printbibliography[heading=none]

\end{document}